\documentclass[a4paper, 11pt]{article}
\usepackage[fleqn]{amsmath}

\usepackage{amssymb,amsfonts,latexsym,wasysym}
\usepackage{amsthm}
\usepackage{enumitem}
\usepackage[dvips,a4paper, hmargin={2.6cm,2.8cm},vmargin={2.8cm,2.8cm}]{geometry}

\usepackage{enumitem}
\setlist{itemsep=0.4pt,topsep=0pt}

\usepackage{xcolor}

\setlist[itemize]{topsep=0pt}
\setlist[enumerate]{topsep=0pt}

\newtheorem{theorem}{Theorem}[section]

\theoremstyle{definition}
\newtheorem{remark}[theorem]{Remark}
\newtheorem*{nota*}{Notation}
\newtheorem{nota}[theorem]{Notation}

\theoremstyle{theorem}
\newtheorem{defi}[theorem]{Definition}

\newtheorem{lemma}[theorem]{Lemma}
\newtheorem{fact}[theorem]{Fact}
\newtheorem{prop}[theorem]{Proposition}
\newtheorem{cor}[theorem]{Corollary}

\renewcommand{\c}{\mathbb{M}}
\renewcommand{\P}{\mathbb{P}}

\newcommand{\tp}{\mathrm{tp}}
\newcommand{\acl}{\mathrm{acl}}
\newcommand{\dcl}{\mathrm{dcl}}

\newcommand{\tpt}{$\mathrm{TP}_2$\,}
\newcommand{\ntpt}{$\mathrm{NTP}_2$\,}
\newcommand{\nsop}{$\mathrm{NSOP}_1$\,}
\newcommand{\sop}{$\mathrm{SOP}_1$\,}
\newcommand{\nip}{$\mathrm{NIP}$\,}
\newcommand{\ov}{\overline}
\newcommand{\bdd}{\mathrm{bdd}}

\def\Ind#1#2{#1\setbox0=\hbox{$#1x$}\kern\wd0\hbox to 0pt{\hss$#1\mid$\hss}
\lower.9\ht0\hbox to 0pt{\hss$#1\smile$\hss}\kern\wd0}
\def\ind{\mathop{\mathpalette\Ind{}}}
\def\Notind#1#2{#1\setbox0=\hbox{$#1x$}\kern\wd0\hbox to 0pt{\mathchardef
\nn=12854\hss$#1\nn$\kern1.4\wd0\hss}\hbox to
0pt{\hss$#1\mid$\hss}\lower.9\ht0 \hbox to
0pt{\hss$#1\smile$\hss}\kern\wd0}

\renewenvironment{proof}{\vspace{-0.25cm}
{\bf Proof} }{\hfill $\Box$}

\begin{document}
\title{Model theory of Steiner triple systems} 

\date{\today}
\author{Silvia Barbina, Enrique Casanovas \thanks{Research partially supported by grant MTM2014-59178P from \textit{Ministerio de Econom\'{\i}a, Industria y Competitividad}, Spain. The first author has been supported by a series of Santander Mobility Scholarships.}}

\maketitle

\begin{abstract} 
A Steiner triple system is a set $S$ together with a collection $\mathcal{B}$ of subsets of $S$ of size~3 such that any two elements of $S$ belong to exactly one element of $\mathcal{B}$.  It is well known that the class of finite Steiner triple systems has a Fra\"{\i}ss\'e limit $M_{\mathrm{F}}$. Here we show that the theory $T^\ast_\mathrm{Sq}$ of $M_{\mathrm{F}}$ is the model completion of the theory of Steiner triple systems. We also prove that  $T^\ast_\mathrm{Sq}$ is not small and it has quantifier elimination, \tpt, \nsop, elimination of hyperimaginaries and weak elimination of imaginaries.
\end{abstract}

\section{Introduction and preliminaries}

A \textit{Steiner triple system} (STS) is a set $A$ together with a set $\mathcal{B}$ of subsets of $A$ of size 3, called \textit{blocks}, such that every two elements of $A$ belong to exactly one element of $\mathcal{B}$. When the set $A$ is finite, we say that the STS is finite; an STS is infinite otherwise. 
It is well known that a necessary and sufficient condition for the existence of an STS of finite cardinality $n$ is that $n \equiv 1 \mbox{ or } 3 \pmod 6$.

The literature on finite Steiner triple systems is vast; \cite{colbournrosa} gives an encyclopedic account of themes and results in the area. On the other hand, far fewer results have been obtained on \textit{infinite} STSs. Until recently, the interest arose in response to questions about automorphism group actions, or in order to construct examples with combinatorial properties that are hard to obtain in the finite case -- for instance, \cite{PJCorbit} gives an orbit theorem for infinite STSs; \cite{thomas}~proves that if $S$ is an infinite STS in which any triangle (a set of three points not in a block) is contained in a finite subsystem, and the automorphism group of $S$ acts transitively on triangles, then $S$ is a projective space over $\mathrm{GF}(2)$ or an affine space over $\mathrm{GF}(3)$; \cite{chicotetal} gives a construction of $2^\omega$ non-isomorphic countable STSs that are uniform and $r$-sparse for $r \geq 4$; in \cite{cameronwebb}, uncountably many non-isomorphic perfect countable STSs are constructed.

The results in this paper are motivated by a model theoretic viewpoint on the countable universal homogeneous locally finite Steiner quasigroup $M_{\mathrm{F}}$, whose existence was first noted in~\cite{thomas}, and which is the Fra\"{\i}ss\'e limit of the class of finite Steiner quasigroups. The interest in \cite{thomas} is permutation-group theoretic: no model theoretic treatment of this Fra\"{\i}ss\'e limit has appeared in the literature so far.

From the point of view of model theory, STSs can be viewed both as relational and as functional structures, in the sense of Definitions~\ref{defSTS} and~\ref{defquasigroup} below. In this paper we distinguish between Steiner triple systems (relational) and \textit{Steiner quasigroups} (functional). We prove that the theory 
of $M_{\mathrm{F}}$ is in fact the model completion of the theory of Steiner triple systems, viewed as functional structures, and we describe several of its properties. 

In Section~\ref{modelcompletion} we give 
 an axiomatisation of the class of existentially closed Steiner quasigroups, and we show that the resulting theory $T^\ast_\mathrm{Sq}$ is complete, has quantifier elimination, and it is the model completion of the theory of all Steiner quasigroups. In Section~\ref{smallness} we show that $T^\ast_\mathrm{Sq}$ is not small, and in Section~\ref{fraisse} we show that the Fra\"{\i}ss\'e limit of the class of finite Steiner quasigroups is a prime model of $T^\ast_\mathrm{Sq}$. 
We then give a characterisation of algebraic closure and we prove that $T^\ast_\mathrm{Sq}$ eliminates the quantifier $\exists^\infty$. In Section~\ref{amal} we prove certain results concerning amalgamation and joint consistency of formulas. These results are used in Section~\ref{tp2andnsop1} to classify $T^\ast_\mathrm{Sq}$ in terms of the dividing lines of first order theories: $T^\ast_\mathrm{Sq}$ is a new example of a theory with TP$_2$ and NSOP$_1$. In Section~\ref{section8} we use the approach developed in \cite{Conant17} to show that $T^\ast_\mathrm{Sq}$ has elimination of hyperimaginaries and weak elimination of imaginaries.

As an incidence structure whose theory is \tpt and \nsop, the structure $M_F$ in this paper is an interesting counterpart to the existentially closed incidence structures omitting the complete incidence structure $K_{m,n}$ in \cite{ConantKruckman17}. Rather like $T^\ast_\mathrm{Sq}$, the theories $T_{m,n}$ in  \cite{ConantKruckman17} are also \tpt and properly \nsop when $m,n \geq 2$, and they do not have a countable saturated model. A notable difference is that the existence of a prime model for $T_{m,n}$ is unknown, and shown in \cite{ConantKruckman17} to be a necessary condition for a positive answer to the open problem of whether every finite model of $T_{m,n}^p$ embeds in a finite model of $T_{m,n}^c$. The analogous property for STSs is well known and it ensures the joint embedding and amalgamation properties for the class of finite STSs, and hence the existence of the Fra\"{\i}ss\'e limit $M_F$, which turns out to be a prime model of its theory. 

Steiner triple systems are a subclass of the class of Steiner systems, which are defined similarly: in a Steiner system $S(t,k,n)$, the underlying set has size $n$, blocks are subsets of size $k$ and each $t$-element subset (for $t < k$) is contained in exactly one block. In a recent preprint \cite{balpao}, Baldwin and Paolini use a version of an amalgamation technique due to Hrushovski to construct $2^{\aleph_0}$ strongly minimal Steiner systems with blocks of size $k$ for every integer $k$. These provide counterexamples to Zilber's trichotomy conjecture with 
 a natural combinatorial characterization that is independent of the conjecture. By contrast, standard Fra\"{\i}ss\'e amalgamation gives a structure, our $M_F$, that is a foremost example to consider model-theoretically, and sits at the opposite end of the model theoretic spectrum from its counterparts in \cite{balpao}. Another striking difference is that in $T^\ast_\mathrm{Sq}$  algebraic closure coincides with definable closure. We refer the reader to \cite{balpao}, Remark~6.1 for a full comparison between the strongly miminal theories in~\cite{balpao} and $T^\ast_\mathrm{Sq}$. 
 
As mentioned, Steiner triple systems can be described both as relational and as functional structures. The choice of language determines substructures, and so, in particular, it is relevant to amalgamation.
\begin{defi} \label{defSTS}A \textbf{Steiner triple system} (STS) is a relational structure $(A,R)$  where $R$ is a ternary relation on a set $A$  such that
\begin{enumerate}
\item if  $R(a,b,c)$ then   $R(\sigma(a),\sigma(b),\sigma(c))$  for every permutation $\sigma$ of $\{a,b,c\}$;
\item  $R(a,a,b)$ iff  $a=b$;
\item for every two different $a,b\in A$ there is a unique $c$ such that $R(a,b,c)$.
\end{enumerate}
A structure $(A, R)$ is a \textbf{partial Steiner triple system}  if instead of 3  we require that for every two different $a,b\in A$ there is at most one $c\in A$ such that $R(a,b,c)$.
\end{defi}

\begin{defi} \label{defquasigroup} A \textbf{Steiner quasigroup} is a structure $(A,\cdot)$  where $\cdot$ is a binary operation on $A$ such that
\begin{enumerate}
\item $a\cdot b= b\cdot a$
\item $a\cdot a =a$
\item $a\cdot (a\cdot b)=b \,$.
\end{enumerate}
\end{defi}
Thus, in a Steiner triple system $(A, R)$ three distinct points $a, b$ and $c$ form a block if and only if $R(a,b,c)$ holds; in a Steiner quasigroup three distinct points form a block if and only if each of them is the product of the other two.

Steiner triple systems and Steiner quasigroups are essentially the same objects in the following sense.
\begin{itemize}
\item Let $(A,R)$ be a Steiner triple system and  define a binary operation $\cdot$ on $A$   as follows:  $a\cdot b$ is the unique $c\in A$ such that  $R(a,b,c)$.  Then $(A,\cdot)$ is a Steiner quasigroup.
\item Let   $(A,\cdot)$ be  a Steiner quasigroup  and let $R$ be the graph of the operation $\cdot$,  that is,
$R(a,b,c) \mbox{ iff } a\cdot b=c $. Then  $(A,R)$ is a Steiner triple system.
\end{itemize}

The correspondence between STSs and Steiner quasigroups means that, in practice, we often switch from relational to functional terminology when convenient, and we sometimes refer to the \textit{product} of two elements in a Steiner triple system.  
\begin{defi} Let $(A,R)$ be a partial STS, and let $a, b \in A$. We say that $a,b\in A$ have a \textbf{defined product} in $A$  if there is $c \in A$ such that $R(a,b,c)$. When this is the case, $c$ is said to be the \textbf{product} of $a$ and $b$.
\end{defi}

It is well known that a  finite partial STS can always be embedded in a finite STS (where embeddings are understood in the model theoretic sense, 
 so the blocks in the image of a partial STS under an embedding are the images of the  blocks in the original partial STS). This can be done in a number of different ways -- see, for example, \cite{treash}, \cite{andersenetal} and \cite{lindner}. For the purposes of this paper, the specific constructions are not relevant and it is enough to state the general result below.
\begin{fact}\label{F1} \begin{enumerate}
\item  Every  partial Steiner triple system of infinite cardinality $\kappa$ can be embedded in a Steiner triple system of cardinality $\kappa$.  
\item Every  partial finite Steiner triple system  can be embedded in a finite Steiner triple system.
 \end{enumerate}

\end{fact}
 \begin{proof} \ 1. Suppose $\kappa$ is an infinite cardinal, and let $(A, R)$ be a partial STS of cardinality $\kappa$. We define a chain $\{(A_i, R_i) \, |\,  i < \omega \}$ of partial STSs, where $(A_0, R_0) = (A, R)$, and $(A_{i+1}, R_{i+1})$ is obtained as follows: for every (unordered) pair $\{a,b \}$ of elements of $A_i$ that do not have a defined product in $A_i$, if $a \neq b$ add a \textit{new} element $a \cdot b \notin A_i$ and put $R_{i+1}(a, b, a \cdot b)$. Define $R_{i+1}$ consistently on all permutations of $\{a, b, a\cdot b\}$. If $a= b$, put $R_{i+1}(a, a, a)$. 
Let $(B, S) = (\bigcup_{i \in \omega} A_i, \bigcup_{i\in \omega}R_i)$. It is easy to see that $(B, S)$ is an STS.
 
2. See, for example, Theorem~1 in \cite{andersenetal}.
\end{proof}

The next lemma is an immediate consequence of Fact~\ref{F1}. It is stated for Steiner quasigroups, rather than for Steiner systems, because a substructure (in the model theoretic sense) of an STS is a \textit{partial} STS, and amalgamation of STSs over a common partial STS is not possible in general. 
\begin{lemma} \label{amalclass}  The class of all Steiner quasigroups has the amalgamation property (AP) and the joint embedding property (JEP).  Likewise, the class of all finite Steiner quasigroups has AP and JEP.
\end{lemma}
\begin{proof} For JEP, use the fact that the disjoint union of two Steiner quasigroups, described in a relational language, is a partial STS, and Fact~\ref{F1}.

For AP, use the fact that the union of two Steiner quasigroups over a common subquasigroup, described in a relational language, is a partial STS, and Fact~\ref{F1}.
\end{proof}

\section{Model completion} \label{modelcompletion}

The  class of all Steiner quasigroups is elementary: its theory, which we denote by $T_{\mathrm{Sq}}$,  has the three universal sentences in Definition~\ref{defquasigroup} as its axioms. In this section we show that the class of existentially closed Steiner quasigroups is elementary. The resulting theory is complete, has quantifier elimination, and it is the model companion of $T_{\mathrm{Sq}}$.

As we have observed, in general the disjoint union of two STSs over a common substructure is not a partial STS, because pairs may arise with more than one product. The next definition specifies conditions on a common substructure which ensure that the disjoint union over that substructure is a partial STS. 

\begin{defi} Let $(B,R)$ be a partial  STS. We say that $A\subseteq B$  is \textbf{relatively closed} in $(B,R)$  if  for every $a,b\in A$ and $c\in B$,  if  $R(a,b,c)$, then $c\in A$. In other words, when two elements of $A$ have a product in $B$, the product belongs to $A$.

Let $(A,R)$ and $(B,S)$  be partial STSs  and let  $C\subseteq A\cap B$.  We say that  $(A,R)$ is \textbf{compatible} with  $(B,S)$ on $C$  if whenever $a,b\in C$  have a defined product $c$ in $(A,R)$, then either they have the same product in $(B,S)$ or they do not have a defined product in $(B,S)$.
\end{defi}
Clearly, if $(A,R)$ is compatible with $(B,S)$ on $C \subseteq A \cap B$, then  $(B,S)$ is compatible with $(A,R)$ on $C$, and we simply say that $(A,R)$ and $(B,S)$ are compatible on $C$.

The next lemma describes cases where the union of two partial STSs is a partial STS. 

\begin{nota} If $(B,S)$ is a partial STS and $A \subseteq B$, then $S^A$ denotes the restriction of the relation $S$ to $A$. \end{nota}

\begin{lemma}\label{mc1}\begin{enumerate}
\item  Let $(A,R)$ and $(B,S)$   be  partial STSs  that are compatible on $A\cap B$.  Then   $(A\cup B, R\cup S)$  is a partial  STS. If, moreover, $S^{A\cap B}\subseteq R^{A\cap B}$, then  $(A,R)\subseteq (A\cup B, R\cup S)$.
\item Assume $(B,R)$   and $(C,S)$  are partial STS and $A=B\cap C$ is relatively closed in $(B,R)$.  If  $(A,R^A)$ and $(C,S)$  are compatible on $A$, then  also  $(B,R)$ and $(C,S)$ are compatible on $A$, and therefore $(B\cup C, R\cup S)$  is a partial  STS.
 \end{enumerate}
\end{lemma}
\begin{proof} 1.\ We must show that for $a, b \in A \cup B$ there is at most one $c \in A \cup B$ such that $R(a,b,c)$ or $S(a,b,c)$. The nontrivial case is when $a$ and $b$ are both in $A \cap B$. Since $(A,R)$  is compatible with $(B,S)$ on $A \cap B$, if $a$ and $b$ have a defined product in $(A, R)$, then they have the same defined product in $(B,S)$ and hence in $(A \cup B, R \cup S)$, or they do not have a defined product in $(B,S)$.

2.\ Suppose that $a, b \in A$ have a defined product $c$ in $(B, R)$. Since $A$ is relatively closed in $(B,R)$, we have $c \in A$. Since $(A, R^A)$ and $(C,S)$ are compatible on $A$, we have that $R(a,b,c)$ implies $S(a,b,c)$ and therefore $a$ and $b$ have the same defined product in $(B,R)$ and in $(C,S)$.
\end{proof}

The proof of the next lemma is similar in flavour to that of Lemma~\ref{mc1} and it is left to the reader.
\begin{lemma}\label{mc2}\begin{enumerate}
\item Let $\{(A_i,R_i)\mid i\in I\}$ be a family of partial STS.  If  $(A_i\cup A_j,R_i\cup R_j)$  is a partial STS for every $i,j\in I$,  then $(\bigcup_{i\in I} A_i,\bigcup_{i\in I} R_i)$  is a partial STS. If $j\in I$  and  $(A_j,R_j)\subseteq (A_i\cup A_j, R_i\cup R_j)$  for every $i\in I$, then  $(A_j,R_j)\subseteq (\bigcup_{i\in I}A_i, \bigcup_{i\in I}R_i)$.
\item Let $\{(A_i,R_i)\mid i\in I\}$ be a family of partial STS  with  common intersection $A=A_i\cap A_j$  for every two different $i,j\in I$.  Assume that $A$ is relatively closed in every $(A_i,R_i)$.
Then $(\bigcup_{i\in I}A_i,\bigcup_{i\in I}R_i)$  is a partial STS. 
\end{enumerate}
\end{lemma}
We can now define the formulas that we use to axiomatise the class of existentially closed Steiner quasigroups.
\begin{defi}
Let $(A,R)$  be a finite partial STS, let $n=|A|$,  and let $A=\{a_1,\ldots,a_n\}$.  We define $\delta_{(A,R)}(x_1,\ldots,x_n)$  to be the conjunction of $\bigwedge_{1\leq i<j\leq n} x_i\neq x_j$  with the positive diagram of $(A,R)$  (with $x_i$ corresponding to $a_i$) written in the product language $L=\{\cdot\}$ of quasigroups, that is, the  conjunction of   all formulas of the form  $x_i\cdot x_j=x_k$  such that  $R(a_i,a_j,a_k)$.

Now let $(B,S)$ be  a  finite partial STS such that $B\supseteq A$  and $A$ is relatively closed in $(B,S)$.  Let  $n+m=|B|$  and  $B=\{a_1,\ldots,a_n,b_1,\ldots ,b_m\}$, and  consider  the  formula \[\delta_{(B,S)}(x_1,\ldots,x_n,y_1,\ldots,y_m),\]
 as defined above, for $(B,S)$, where $x_i$ corresponds to $a_i$  and $y_i$ to $b_i$.  To the pair $((B,S),A)$ we associate the  $L$-sentence
\[\forall x_1\ldots x_n\, (\delta_{(A,S^A)}(x_1,\ldots,x_n)\rightarrow \exists y_1\ldots y_m\, \delta_{(B,S)}(x_1,\ldots,x_n,y_1,\ldots,y_m))\,\]
and we define $\Delta$ as the set of all sentences of this form as $((B,S), A)$ ranges over all pairs where $(B,S)$ is a finite partial STSs and $A$ is a relatively closed subset of $B$.
\end{defi}

\begin{prop}\label{mc3}  A Steiner quasigroup  is existentially closed in the class of all Steiner quasigroups if and only if it is a model of $\Delta$.
\end{prop}
\begin{proof}  Let  $(M,\cdot)$ be an existentially closed Steiner quasigroup. We check that all the sentences in $\Delta$ hold in $(M, \cdot)$.  Let  $(B,S)$ be a  partial STS  and $A$ a  relatively closed subset, with  $|A|=n$,  $A=\{a_1,\ldots,a_n\}$, $|B|=n+m$ and $B=\{a_1,\ldots,a_n,b_1,\ldots,b_m\}$.  Assume $(M,\cdot)\models \delta_{(A,R)}(a_1^\prime,\ldots,a_n^\prime)$, where $R= S^A$.   Define $R^\prime$ on  $A^\prime=\{a_1^\prime,\ldots, a_n^\prime\}$  in such a way that the mapping $a_i\mapsto a_i^\prime$ is an isomorphism of partial STSs.  Choose $b_1^\prime,\ldots,b_m^\prime\not\in M$ and a corresponding relation $S^\prime$  on  $B^\prime=\{a_1^\prime,\ldots,a_n^\prime,b_1^\prime,\ldots,b_m^\prime\}$ so that the mapping defined by  $a_1\mapsto a_i^\prime$  and  $b_j\mapsto b_j^\prime$  is an isomorphism of partial STSs.  Then  $A^\prime$ is  relatively closed in  $(B^\prime,S^\prime)$.  Let $(M,P)$  be the STS associated to $(M,\cdot)$ -- that is, $P$ is the graph of the product in $M$.  By Lemma~\ref{mc1}, we have that $(M,P)$ and $(B^\prime,S^\prime)$ are compatible on $A^\prime =M\cap B^\prime$.
 Then  $(M\cup B^\prime, P\cup S^\prime)$  is a partial STS and so it can be extended to a Steiner triple system  $(N,P^\prime)$. The associated Steiner quasigroup $(N,\cdot)$ is an extension of $(M,\cdot)$, and  $(N,\cdot)\models \delta_{(B,S)}(a_1^\prime,\ldots,a_n^\prime,b_1^\prime,\ldots,b_m^\prime)$.  Since $(M,\cdot)$ is existentially closed in $(N,\cdot)$, we have  $(M,\cdot)\models \exists y_1\ldots y_m \, \delta_{(B,S)}(a_1^\prime,\ldots,a_n^\prime,y_1,\ldots,y_m)$, as required.

Now assume that $(M,\cdot)\subseteq (N,\cdot)$ are  Steiner quasigroups and that  $(M,\cdot)\models \Delta$, and let us check that $(M,\cdot)$ is existentially closed in $(N,\cdot)$.  Let $\varphi(x_1,\ldots,x_n,y_1,\ldots,y_m)$  be a quantifier-free formula and let $a_1,\ldots,a_n\in M$. Assume that
\[(N,\cdot)\models \exists y_1\ldots y_m\, \varphi(a_1,\ldots,a_n,y_1,\ldots,y_m)\,.\]
We want to find $b_1,\dots,b_m\in M$  such that $(M,\cdot)\models \varphi(a_1,\ldots,a_n,b_1,\ldots,b_m)$. We may assume that $\varphi$ is a conjunction of equalities and inequalities between terms of the form $t(x_1,\ldots,x_n,y_1,\ldots,y_m)$. It is easy to find a formula $\psi(x_1,\ldots,x_n,y_1,\ldots,y_m,z_1,\ldots,z_k)$ which is a conjunction of equalities of the form $u\cdot v=w$ and inequalities of the form $u\neq v$ for variables $u,v,w$,   and  such that   $\varphi(x_1,\ldots,x_n,y_1,\ldots,y_m)$ is logically equivalent to 
\[\exists z_1\ldots z_k\, \psi(x_1,\ldots,x_n,y_1,\ldots,y_m,z_1,\ldots,z_k)\, .\] 
So we may assume that $\varphi$ is a formula with this property and forget $\psi$. Choose  $b_1,\ldots,b_m\in N$  such that $(N,\cdot)\models \varphi(a_1,\ldots,a_n,b_1,\ldots,b_m)$. Without loss of generality,  $a_1,\ldots,a_n,b_1,\ldots,b_m$ are pairwise different and $b_1,\ldots,b_m\not \in M$.  Let  $S$  be the graph of  the product $\cdot$ of  $N$  
and let $B=\{a_1,\ldots,a_n,b_1,\ldots,b_m\}$. Then $(B,S^B)$  is a  finite partial STS  and   $A=\{a_1,\ldots,a_n\}$ is relatively closed in $(B,S^B)$,  and we have a corresponding axiom in $\Delta$, which holds in $(M,\cdot)$. Notice that  $(M,\cdot)\models \delta_{(A,S^A)}(a_1,\ldots,a_n)$, so  
\[(M,\cdot)\models \exists y_1\ldots y_m \, \delta_{(B,S^B)}(a_1,\ldots,a_n,y_1,\ldots,y_m)\, ,\] 
and we may choose $b_1^\prime,\ldots,b_m^\prime\in M$  such that $(M,\cdot)\models \delta_{(B,S^B)}(a_1,\ldots,a_n,b_1^\prime,\ldots, b_m^\prime)$. If  $\varphi$  contains an equality of the form  $x_i\cdot y_j= y_k$, then  $a_i\cdot b_j =b_k$, and then  $S(a_i,b_j,b_k)$  and the  equation $x_i\cdot y_j= y_k$  belongs to $\delta_{(B,S^B)}$. Similarly for other kinds of equalities in $\varphi$.  Hence $(M,\cdot)\models \varphi(a_1,\ldots,a_n,b_1^\prime,\ldots,b_m^\prime)$.
\end{proof}

\begin{prop}\label{mc4}  The class  $K^\ast_\mathrm{Sq}$  of existentially closed Steiner quasigroups is elementary. An axiomatization of its theory $T^\ast_\mathrm{Sq}$ is obtained by adding to $ \Delta$ the three axioms of the theory $T_{\mathrm{Sq}}$ of all Steiner quasigroups.   $T^\ast_\mathrm{Sq}$ is a complete theory with elimination of quantifiers,  and it is the model completion of the theory  $T_{\mathrm{Sq}}$ of all Steiner quasigroups. 
\end{prop}
\begin{proof} The first claim is Proposition~\ref{mc3}. Since the axioms of $T_{\mathrm{Sq}}$ are universal, every Steiner quasigroup can be extended to an existentially closed Steiner quasigroup. Hence $T^\ast_\mathrm{Sq}$ is the model companion of $T_{\mathrm{Sq}}$. Since $T_{\mathrm{Sq}}$ has the JEP, 
$K^\ast_\mathrm{Sq}$ has  JEP too  and, by model-completeness, $T^\ast_\mathrm{Sq}$ is a complete theory.  Since $T_{\mathrm{Sq}}$ has AP,  $T^\ast_\mathrm{Sq}$ is the model completion of $T_{\mathrm{Sq}}$  and has quantifier elimination.  See~\cite{Cha-Kei90}, Propositions~3.5.11,~3.5.18 and~3.5.19, for details.
\end{proof}

Since $T^\ast_\mathrm{Sq}$  is a complete theory, it has a monster model.  As usual, the models of $T^\ast_\mathrm{Sq}$ will be identified with small elementary submodels of the monster.
\begin{nota} In the rest of the paper, the monster model of $T^\ast_\mathrm{Sq}$ will be denoted by $(\c_\mathrm{Sq},\cdot)$, and  $\P$ will denote the graph of the product in $(\c_\mathrm{Sq},\cdot)$.
\end{nota}

\begin{remark}\label{mc5} \begin{enumerate}
\item[(i)] Every partial STS of cardinality at most $|\c_\mathrm{Sq}|$ can be embedded in~$(\c_\mathrm{Sq},\P)$.
\item[(ii)] Let $A\subseteq \c_\mathrm{Sq}$ be the universe of a small substructure, and let $R=\P^A$ be the graph of the product on $A$.  If  $(B,S)$ is a partial STS such that $|B| \leq |\c_\mathrm{Sq}|$ and $(A,R)\subseteq (B,S)$,  then there is an embedding of $(B,S)$ into  $(\c_\mathrm{Sq},\P)$ over $A$.
\end{enumerate}
\end{remark}

 \section{Smallness} \label{smallness}
 We show how to construct a finitely generated countable Steiner quasigroup $M$ that embeds every member of a given family of finite Steiner quasigroups. We do this in such a way that the only finite quasigroups that embed in $M$ are the members of the family and their substructures. This construction is used to 
 show that $T^\ast_\mathrm{Sq}$ is not a small theory.

\begin{prop}\label{sma1} Let $\{(A_i,\cdot)\mid i<\omega\}$ be a family of finite Steiner quasigroups such that $|A_i| \geq 3$ for at least one $i \in \omega$.  Then there is a countable infinite quasigroup $(M,\cdot)$ such that:
\begin{enumerate}
\item $M$ is generated by three elements;
\item every $(A_i,\cdot)$ embeds in $(M,\cdot)$;
\item if a finite quasigroup embeds in $(M,\cdot)$, it embeds in some $(A_i,\cdot)$.
\end{enumerate}
\end{prop}
\begin{proof} $M$ is the result of a free construction which is carried out inductively over $\{ A_i\mid i<\omega\}$.  We may assume that the quasigroups  $A_i$  are pairwise disjoint.  Let  $R_i$ be the graph of the product on $A_i$, so $(A_i,R_i)$ is a finite STS. We construct an ascending chain   $\{(B_i,S_i)\mid i<\omega\}$  of finite partial  STS  $(B_i,S_i)$  such that   
\begin{enumerate}
\item $B_0$ has three elements (not in a block);
\item every two  $a,b\in B_i$  have a defined product in $(B_{i+1},S_{i+1})$;
\item  $(A_i,R_i)$ is a substructure of $(B_{i+1},S_{i+1})$;
\item every $a\in B_{i+1}$ can be written as a product of elements of $B_{i}$;
\item if a finite STS $(A,R)$ embeds in $(B_{i+1},S_{i+1})$, then it embeds in some $(A_j,R_j)$  with  $j\leq i$.
\end{enumerate}
Then we take  $M=\bigcup_{i<\omega}B_i$  and  $S=\bigcup_{i<\omega}S_i$. It follows from 2 that $(M,S)$ is an STS. If $(M,\cdot)$ is the corresponding Steiner quasigroup, then $M$ is generated by the three elements of $B_0$ and it has the required properties.

Let $(B_0,S_0)$ be such that $|B_0|=3$ and $S_0 = \{(b,b,b) : b \in B\}$ (so $B_0$ is the partial STS with three elements that do not form a block). Now assume that $(B_i,S_i)$  has been constructed and that it contains three elements with no product defined among them. Assume $(A_i,R_i)$ is generated by  $a_1,\ldots,a_k$.  
We extend $(B_i, S_i)$ to a partial STS $(B^{(1)}_i,S^{(1)}_i)$ by adding a product $ a \cdot b$ for each pair $\{a,b\}$  of elements of $B_i$ whose product is not defined in $(B_i,S_i)$, in such a way that different pairs have different products. 
We iterate this procedure until we obtain a partial STS $(B^{(n)}_i, S^{(n)}_i)$, where $n$ depends on $i$, that contains a subset of size  $2k+3$, say $\{b_1,\ldots,b_{2k+3}\}$,  with no product defined among its elements.  We may assume that $A_i\cap B^{(n)}_i=\emptyset$. Define $(B_{i+1},S_{i+1})$  as the common extension of $(A_i,R_i)$ and $(B^{(n)}_i,S^{(n)}_i)$   with universe $B_{i+1}=B^{(n)}_i\cup A_i$  and where $S_{i+1}$ is obtained by adding to  $R_i\cup S^{(n)}_i$  the products corresponding to all triples  of the form  $\{b_i,b_{k+i},a_i\}$  for $i=1,\ldots k$, as well as all the necessary triples of the form $(a,a,a)$. Note that no product is defined among $b_{2k+1},b_{2k+2},b_{2k+3}$  and that every element of $A_i$ is now obtained as an iterated product of elements of $B_i$.

If $(A,R)$ is a finite STS of cardinality $\leq 3$,  then $(A,R)$ embeds in any  $(A_i,R_i)$ such that $|A_i|\geq 3$.   Let $(A,R)$  be a finite STS  which is a substructure of  $(B_{i+1},S_{i+1})$, assume that $|A|>3$ and suppose that $A$ is not contained in any $A_j$  with  $j < i$.  We may assume inductively that  $A\not\subseteq B_i$.  Suppose for a contradiction that $A \nsubseteq A_i$. If  $A\cap A_i\neq \emptyset$, take $a\in A\cap A_i$  and  $b\in A\smallsetminus A_i$ (so $b \in B^{(n)}_i$). There is a defined product $a\cdot b$ in $(A,R)$. Notice that $a\cdot b\in A\smallsetminus A_i$ (so $a \cdot b \in B^{(n)}_i$). Then 
$b, a\cdot b\in B_{i+1}$ have a product  $a= b\cdot (a\cdot b)\in A_i$.  By construction, there is a list $a_1,\ldots,a_k$  of generators of $A_i$  and a list  $b_1,\ldots,b_{2k+3}$  of elements of $B_i^{(n)}$  without defined product in $B_i^{(n)}$   and there is some $j\leq k$  such that  $a=a_j$,  $b=b_j$   and  $a\cdot b= b_{k+j}$.   Since there is a unique element of $A_i$ whose product with $b=b_j$ is defined in $B_{i+1}$, it follows that  $A\cap A_i=\{a\}$.  Since $|A|>3$, there is some $c\in A$ different from $a,b$ and $a\cdot b$, hence with a defined product with $a$. Since in $B_{i+1}$ the only defined products of $a$  with elements not in $A_i$ are the products with $b$ and with $a\cdot b$, it follows that $c \in A_i$, contradicting 
$A \cap A_i = \{a \}$.

If $A\cap A_i =\emptyset$, then $A \subseteq B^{(n)}_i$ and there is some $a\in A$  such that  $a \in B^{(n)}_i\smallsetminus B_i$.  The elements of $B_i^{(n)} \smallsetminus B_i$ have been obtained  iteratively as products of previous pairs in a free way. We may assume that $a \in B^{(n)}_i \setminus B^{(n-1)}_i$, that is, no element of $A$ has been obtained after obtaining $a$.   Since $|A|>3$, there are different  $b,c\in A\smallsetminus \{a\}$   with   $c\neq a\cdot b$.  By our choice of $a$,  the elements $b$ and $a\cdot b$ are in $B^{(n-1)}_i$, and $a$ is the product of $b$ and $a\cdot b$, a pair without a defined product in $B^{(n-1)}_i$. But, similarly, $a$ is the product of $c$ and $a\cdot c$, elements without a defined product in $B^{(n-1)}_i$.  In this case, the pairs  $\{b,a\cdot b\}$ and  $\{c,a\cdot c\}$ coincide. But this is not possible, since $c\neq b$  and  $c\neq a\cdot b$.
\end{proof}

The following result by Doyen gives a countable family of finite Steiner quasigroups none of which embeds in another member of the family. Applying the construction of Proposition~\ref{sma1} to this family gives uncountably many complete 3-types over $\emptyset$.
\begin{lemma}[Doyen] \label{doyen} For all $ n \equiv 1, 3 \pmod{6}$ there is an STS of cardinality $n$ that does not embed any STS of cardinality $m$  for $3<m<n$. 
\end{lemma}
\begin{proof} \cite{Doyen69}. \end{proof}

\begin{theorem} \label{notsmall} $T^\ast_\mathrm{Sq}$ is not small. In fact, there are $2^\omega$ complete types over $\emptyset$ in three variables.
\end{theorem}
\begin{proof} For $n \equiv 1, 3 \pmod{6}$, let $(A_n,R_n)$ be the STS of cardinality $n$ given by Lemma~\ref{doyen}, so $A_n$ does not embed any STS of cardinality $m$  for  $3<m<n$.  Let $(A_n,\cdot)$ be the corresponding Steiner quasigroup.  Let  $I$ be the set of all natural numbers $n$ such that  $n\equiv \,1,3 \pmod{6}$. For every infinite subset $X\subseteq I$,  let $(M_X,\cdot)$  be  the countable Steiner quasigroup obtained from the family $\{(A_n,\cdot)\mid n\in X\}$  as in Proposition~\ref{sma1}. Then $M_X$ is generated by three elements and the only non-trivial finite Steiner quasigroups embeddable in $(M_X,\cdot)$  are the quasigroups $(A_n,\cdot)$  with  $n\in X$. 
Clearly, if $X\neq Y$, then $(M_X,\cdot)$ and $(M_Y, \cdot)$ are not isomorphic. Since $(M_X,\cdot)$ embeds in the monster model $(\c^\ast_\mathrm{Sq},\cdot)$ of $T^\ast_\mathrm{Sq}$, we may assume that $M_X\subseteq \c^\ast_\mathrm{Sq}$.  Choose three generators  $a,b,c\in \c^\ast_\mathrm{Sq}$ of $M_X$  and let $p_X(x,y,x)=\tp(a,b,c)$.  Then $p_X(x,y,z)\neq p_Y(x,y,z)$  if $X\neq Y$.  This gives $2^\omega$ 3-types over $\emptyset$.
\end{proof}

\section{The Fra\"{\i}ss\'e limit} \label{fraisse}

The existence of the Fra\"{\i}ss\'e limit of all finite Steiner quasigroups is well known: the limit is the countably
infinite homogeneous locally finite Steiner quasigroup \cite{PJCinfinitels, thomas}.

\begin{fact}\label{f1} The class $\mathrm{K_{\mathrm{Sq}}^{\mathrm{fin}}}$ of all finite Steiner quasigroups has a Fra\"{\i}ss\'e limit  $(M_\mathrm{F},\cdot)$,  the unique (up to isomorphism) countable ultrahomogeneous Steiner quasigroup  whose age is $K_{\mathrm{Sq}}^{\mathrm{fin}}$. Moreover,  $(M_\mathrm{F},\cdot)$ is locally finite.
\end{fact}
\begin{proof} By Fact~\ref{F1}, the class $K_{\mathrm{Sq}}^{\mathrm{fin}}$ has the amalgamation property and the joint embedding property. It is clear that $K_{\mathrm{Sq}}^{\mathrm{fin}}$ has the hereditary property and that it contains only countably many isomorphism types. Since $(M_\mathrm{F},\cdot)$ is the union of a countable ascending chain of finite structures, every finitely generated substructure of $(M_\mathrm{F},\cdot)$ is finite.
\end{proof}

The next corollary follows from the properties of the Fra\"{\i}ss\'e limit and from Fact~\ref{F1}.
\begin{cor}\label{f2} Let $P_\mathrm{F}$ be the graph of the product of the Fra\"{\i}ss\'e limit $(M_\mathrm{F},\cdot)$. \begin{enumerate}  
\item Every finite partial STS  can be embedded in $(M_\mathrm{F},P_\mathrm{F})$.
\item Assume that $A\subseteq M_\mathrm{F}$ is the universe of a finite substructure and $R=P_\mathrm{F}^A$ is the graph of the product on $A$.  If  $(B,S)$ is a finite partial STS  such that  $(A,R)\subseteq (B,S)$,  then there is an embedding of $(B,S)$ into  $(M_\mathrm{F},P_\mathrm{F})$ over $A$.
\end{enumerate}
\end{cor}

The next two propositions show that the Fra\"{\i}ss\'e limit $M_\mathrm{F}$ is existentially closed, and it is a prime model of $T^\ast_\mathrm{Sq}$.
\begin{prop}  The Fra\"{\i}ss\'e limit $(M_\mathrm{F},\cdot)$  is a model of $T^\ast_\mathrm{Sq}$, the model completion of the theory $T_{\mathrm{Sq}}$ of all Steiner quasigroups. 
\end{prop}
\begin{proof} We check that $(M_\mathrm{F},\cdot)$  satisfies the axioms in $\Delta$.  Recall that $P_\mathrm{F}$ is the graph of the product of $M_\mathrm{F}$. Let  $(B,S)$ be a finite partial STS  and  $A\subseteq B$  a relatively closed subset. Let $n=|A|$,  $n+m=|B|$,  $A=\{a_1,\ldots,a_n\}$, $B=\{a_1,\ldots,a_n,b_1,\ldots,b_m\}$ and $R=S^A$. Assume $(M_\mathrm{F},\cdot)\models \delta_{(A,R)}(a_1^\prime,\ldots,a_n^\prime)$.  Define $R^\prime$ on $A^\prime=\{a_1^\prime,\ldots,a_n^\prime\}$ in  such a way that  $a_i\mapsto a_i^\prime$  is an isomorphim of partial STSs. Choose $b_1^\prime,\ldots,b_m^\prime\not\in M_\mathrm{F}$ and define $S^\prime$ on $B^\prime=\{a_1^\prime,\ldots,a_n^\prime,b_1^\prime,\ldots,b_m^\prime\}$ in such a way that the mapping defined by $a_i\mapsto a_i^\prime$ and $b_i\mapsto b_i^\prime$  is an isomorphism of partial STSs. Then $R^\prime= {S^\prime}^{A^\prime}$ and $A^\prime$ is relatively closed in $(B^\prime, S^\prime)$. Let $C \subseteq M_\mathrm{F}$ be the universe of the finite substructure generated by $A^\prime$ in $(M_\mathrm{F},\cdot)$ and let $P$ be the graph of the product on $C$, so $P= P_\mathrm{F}^C$. Then $(C,P)$ is a finite STS, $A^\prime= C\cap B^\prime$  and  $(C,P)$  is compatible with  $(B^\prime,S^\prime)$ on $A^\prime$. Therefore $(C\cup B^\prime, P\cup S^\prime)$ is a partial STS. Moreover, $R^\prime\subseteq  P$ and, by Lemma~\ref{mc1},   $(C,P)\subseteq (C\cup B^\prime, P\cup S^\prime)$.  By Lemma~\ref{f2}, there is an embedding \[f:(C\cup B^\prime,P\cup S^\prime)\rightarrow (M_\mathrm{F},P_\mathrm{F})\]
 over $C$.  Then  
 \[(M_\mathrm{F},\cdot)\models \delta_{(B,S)}(a_1^\prime,\ldots,a_n^\prime,f(b_1^\prime),\ldots,f(b_m^\prime))\, .\]
\end{proof}

\vskip 12pt
\begin{prop} The Fra\"{\i}ss\'e limit $(M_\mathrm{F},\cdot)$ is a prime model of $T^\ast_\mathrm{Sq}$.
\end{prop}
\begin{proof} Let $a_1,\ldots,a_n\in M_\mathrm{F}$  and let us prove that $\tp(a_1,\ldots,a_n)$ is isolated.  We may assume than the $a_i$ are pairwise distinct. The substructure of $(M_\mathrm{F},\cdot)$ generated by $a_1,\ldots,a_n$ is finite, say of cardinality $n+m$, and we fix an enumeration  $a_1,\ldots,a_{n+m}$ of it. 

Let  $\varphi(x_1,\ldots, x_{n+m})$   be the conjunction of $\bigwedge_{1\leq i<j\leq  n+m}x_i\neq x_j$  with all the equalities of the form  $x_i\cdot x_j = x_k$  such that $a_i\cdot a_j= a_j$. 
We claim that the formula 
\[\exists x_{n+1}\ldots x_{n+m}\, \varphi(x_1,\ldots,x_{n+m})\]
  isolates $\tp(a_1,\ldots,a_n)$.  Clearly,  the tuple $a_1,\ldots,a_n$  satisfies this formula.  Consider the monster model $(\c_\mathrm{Sq},\cdot)$, an elementary extension of $(M_\mathrm{F},\cdot)$,  and let $b_1,\ldots,b_n$ in $\c_\mathrm{Sq}$ be a tuple such that  
\[(\c_\mathrm{Sq},\cdot)\models \exists x_{n+1}\ldots x_{n+m}\, \varphi(b_1,\ldots,b_n,x_{n+1}\ldots,x_{n+m})\, .\]
 Take $b_{n+1},\ldots,b_{n+m}\in \c_\mathrm{Sq}$  such that $(\c_\mathrm{Sq},\cdot)\models \varphi(b_1,\ldots,b_{n+m})$.  Then  $\{b_1,\ldots,b_{n+m}\}$  is the universe of a substructure of $(\c_\mathrm{Sq},\cdot)$  and the mapping defined by  $a_i\mapsto b_i$  is an isomorphism. By elimination of quantifiers,  $\tp(a_1,\ldots,a_n)=\tp(b_1,\ldots,b_n)$.
\end{proof}

\begin{remark}  By Theorem~\ref{notsmall}, there is no countable saturated model of $T^\ast_\mathrm{Sq}$, and so in particular the Fra\"{\i}ss\'e limit $(M_\mathrm{F},\cdot)$  is not saturated. This can also be seen directly: for example, $M_\mathrm{F}$ does not realise the type of a finitely generated infinite Steiner quasigroup.  
\end{remark}

\begin{remark}  By Theorem~\ref{notsmall}, there are $2^{\omega}$ non-isomorphic countable Steiner quasigroups generated by three elements. Therefore there are uncountably many isomorphism types of finitely generated Steiner quasigroups, and so the class of finitely generated Steiner quasigroups does not have a Fra\"{\i}ss\'e limit.
\end{remark}

\section{Algebraic closure and elimination of $\exists^\infty$}
In this section, all sets and tuples are chosen in $\c_\mathrm{Sq}$, the monster model of $T^\ast_\mathrm{Sq}$.

\begin{defi} \label{rank} Let $t(x_1,\ldots,x_n)$ be a term in the language of Steiner quasigroups  $L=\{\cdot\}$. The \textbf{rank} of $t$ is $m+1$, where $m$   is the number of occurrences of $\cdot$ in $t$. 
\end{defi}
It follows from Definition~\ref{rank} that the terms of rank $1$ are the variables.  Moreover, the rank of $t_1\cdot t_2$ equals the sum of the  ranks of $t_1$ and $t_2$.

\begin{defi}\label{ac0}  Let $A$  be a subset of $\c_\mathrm{Sq}$.  The universe of the substructure of $\c_\mathrm{Sq}$ generated by $A$ will be denoted by  $\langle A\rangle $.  The set $\langle A\rangle_k $  is the subset of $\langle A\rangle$ consisting of the elements that can be written as $t(a_1,\ldots,a_n)$, where $t(x_1,\ldots,x_n)$ is a term of rank $\leq k$ with $a_1,\ldots,a_n\in A$.  Hence,  $\langle A\rangle_1= A$  and  $\langle A\rangle = \bigcup_{k\geq 1} \langle A\rangle_k$.  If $A=\{a_1,\ldots,a_n\}$, we sometimes use the notation $\langle A\rangle =\langle a_1,\ldots,a_n\rangle$  and $\langle A\rangle_k = \langle a_1,\ldots,a_n\rangle_k$.
\end{defi}

\begin{lemma}\label{ac1}  Let $\P$ be the graph of the product in $\c_\mathrm{Sq}$, let $A= \langle  a_1,\ldots,a_n\rangle_m$  and let $B= \langle b_1,\ldots,b_n\rangle_m$.  If the mapping $a_i\mapsto b_i$ extends to an isomorphism between  the partial STSs $(A,\P^A)$  and $(B, \P^B)$  and $\varphi(x_1,\ldots,x_n)$  is a quantifier-free formula all of whose terms have rank $\leq m$, then  $(\c_\mathrm{Sq},\cdot)\models \varphi(a_1,\ldots,a_n)$  if and only if  $(\c_\mathrm{Sq},\cdot)\models \varphi(b_1,\ldots,b_n)$.
\end{lemma}
\begin{proof}  For every term $t(x_1,\ldots,x_n)$ of rank $\leq m$, the element  $t(a_1,\ldots,a_n)$ is in $A$ and it is sent to $t(b_1,\ldots,b_n)\in B$ by the isomorphism that extends $a_i\mapsto b_i$. Therefore $a_1,\ldots ,a_n$  and $b_1,\ldots ,b_n$  satisfy the same equalities between terms of rank $\leq m$. 
\end{proof}

Let $\varphi(y,a_1, \ldots a_n)$ be a quantifier-free formula that describes how an element $y$ is related to a finite partial STS $A=\{a_1, \ldots, a_n\} \subseteq \c_\mathrm{Sq}$. We show that there is a number $k$, which depends on the rank of the terms in $\varphi$, such that whenever $\varphi(y, a_1, \ldots, a_n)$ is satisfied by an element that is algebraic over $(a_1, \ldots, a_n)$ but cannot be written as a term $t(a_1, \ldots, a_n)$ of rank at most $k$, there are arbitrarily many realizations. The idea is that a finite partial STS only determines the behaviour of iterated products of its elements up to a certain rank.
\begin{prop}\label{ac2} Let $\varphi=\varphi(x,x_1,\ldots,x_n)$  be a quantifier-free formula in the language $L=\{\cdot\}$,  and let $m$ be an upper bound for the rank of the terms occurring in $\varphi$.  Let  $\psi_i(x,x_1,\ldots,x_n)$ be the conjunction of all the inequalities of the form $x\neq t(x_1,\ldots,x_n)$  for terms $t$ of rank $\leq i$.  There is  a number $k$, depending only on $n$ and $m$, such that for every $r \in \omega$ the following sentence holds in $\c_\mathrm{Sq}$:
 $$\forall x_1\ldots x_n \,(\exists x \, (\psi_k(x,x_1,\ldots,x_n)\wedge \varphi(x,x_1,\ldots,x_n))\rightarrow \exists^{\geq r} x\,  \varphi(x, x_1,\ldots,x_n))\, .$$
\end{prop}
\begin{proof} Let $k_0$ be larger than the number of terms $t(x,x_1,\ldots,x_n)$  of rank  $\leq m$, and let $k= 2^{k_0}\cdot m$.  Let  $a,a_1,\ldots,a_n\in \c_\mathrm{Sq}$ and assume $\c_\mathrm{Sq}\models \psi_k(a,a_1,\ldots,a_n)\wedge \varphi(a,a_1,\ldots,a_n)$.  Note that $k_0>|\langle a,a_1,\ldots,a_n\rangle_m|$.  First we claim that there is a set $X$  such that  
$$\langle a_1,\ldots,a_n\rangle_m\subseteq X\subseteq \langle a,a_1,\ldots,a_n\rangle_m\cap \langle a_1,\ldots,a_n\rangle_{2^{k_0}\cdot m}$$
and $X$ is relatively closed in  $\langle a,a_1,\ldots,a_n\rangle_m$  (that is, if  $b,c\in X$  and  $b\cdot c\in \langle a,a_1,\ldots,a_n\rangle_m$, then $b\cdot c\in X$). In order to obtain $X$, we build a chain 
$$\langle a_1,\ldots,a_n\rangle_m =X_0\subseteq X_1\subseteq \cdots \subseteq X_i=X\subseteq \langle a,a_1,\ldots,a_n\rangle_m \cap \langle a_1,\ldots,a_n\rangle_{2^{k_0}\cdot m}$$ 
where $X_j\subseteq \langle a_1,\ldots,a_n\rangle_{2^j\cdot m}$ for all $j \leq i$.  The idea is as follows: if no product of elements of $\langle a_1,\ldots,a_n \rangle_m$ belongs to $\langle a,a_1,\ldots,a_n\rangle_m\smallsetminus \langle a_1,\ldots,a_n\rangle_m$ we take $X=\langle a_1,\ldots,a_n\rangle_m$. Otherwise we form $X_1$ by adding to $X_0$ all such products (which are in $\langle a_1,\ldots,a_n\rangle_{2\cdot m}$). We ask again if any products of elements of $X_1$ belong to $\langle a,a_1,\ldots,a_n\rangle_m\smallsetminus X_1$ and we continue in this way. Formally,  
\[X_{j+1} = X_j \cup \{b \cdot c  \mid  b,c \in X_j \mbox{ and } b \cdot c \in \langle a,a_1,\ldots,a_n\rangle_m \smallsetminus X_j \}\, . \]
Let $D=\langle a,a_1,\ldots,a_n\rangle_m$. Since $|D|< k_0$ and $X_j\subseteq D$, there is $i \leq k_0$ such that   $X_i=X_{i+1}$,  and we can take $X=X_i$. By our choice of $k$,  we have $a\not\in X$.  Choose pairwise  disjoint sets $B_1,\ldots,B_r$, each  disjoint from $D $ and of the same cardinality as $D\smallsetminus X$,  and  choose bijections $f_j:D\rightarrow X\cup B_j$ each of which is the identity on $X$.   Define a relation $R_j$ on    each $X\cup B_j$ in such a way that $f_j$ is an isomorphism of partial STS between $(D,\P^D)$  and $(X\cup B_j,R_j)$, where $\P$ is the graph of the product in $\c_\mathrm{Sq}$.  Let $A=\langle a_1,\ldots,a_n\rangle$. Note that $A\cap (X\cup B_j) = X  = (X\cup B_j)\cap (X\cup B_l)$  whenever $j \neq l$.   We claim that 
\[(A\cup B_1\cup\ldots\cup B_r,\P^A\cup R_1\cup \ldots \cup R_r)\]
  is a partial STS that contains $(A,\P^A)$ as a substructure.  The last point follows  easily from the first one since  $R_j^X= \P^X$ for every $j$.  By   Lemma~\ref{mc2},  it is enough  to check that $(A\cup B_j, \P^A\cup R_j)$  and $(X\cup B_j\cup B_l, R_j\cup R_l)$ are partial STSs for every $j,l$.   Since $X$ is relatively closed in  every $(X\cup B_i,R_i)$, we have that $(X\cup B_i\cup B_j, R_i\cup R_j)$ is always a partial STS. Since $R_j^X= \P^X$, we have that $(A\cup B_j, \P^A\cup R_j)$ is a partial STS.  By  Remark~\ref{mc5},  there is an embedding 
  \[g: (A\cup B_1\cup\ldots\cup B_r, \P^A \cup R_1\cup\ldots \cup R_r)\rightarrow (\c_\mathrm{Sq},\P)\]
   over  $A$.   Let $b_j =g(f_j(a))$  for $j=1,\ldots,r$.  Then $b_1,\dots,b_r$ are  different elements of $\c_\mathrm{Sq}$  and, by Lemma~\ref{ac1},  each $b_j$ realizes  $\varphi(x,a_1,\ldots,a_n)$ in $(\c_\mathrm{Sq},\cdot)$.
\end{proof}

Proposition~\ref{ac2} and quantifier elimination give a characterization of algebraic closure in $\c_\mathrm{Sq}$.
\begin{cor}\label{ac3} For any set $A\subseteq \c_\mathrm{Sq}$, the algebraic closure of $A$ is  the universe   of the substructure generated by $A$, that is,
$\acl(A)=\langle A\rangle \, .$
\end{cor}
\begin{proof} By elimination of quantifiers, Proposition~\ref{ac2} implies that if $a\not\in\langle A\rangle$, then  $a\not\in\acl(A)$.
\end{proof}

\begin{cor}\label{ac4} $T^\ast_\mathrm{Sq}$  eliminates $\exists^\infty$, that is, for each formula $\varphi(x,x_1,\ldots,x_n)$  there is a formula $\psi(x_1,\ldots,x_n)$  defining the set of tuples $(a_1,\ldots,a_n)$ for which 
\[\{b\in \c_\mathrm{Sq}\mid (\c_\mathrm{Sq},\cdot)\models \varphi(b,a_1,\ldots,a_n)\} \]
  is infinite.
\end{cor}
\begin{proof}  By quantifier elimination we may assume that $\varphi$ is quantifier-free.  Choose $k$ and $\psi_k(x,x_1,\ldots,x_n)$  as in Proposition~\ref{ac2} for $\varphi$.  Then $\exists x\, (\varphi(x,x_1,\ldots,x_n)\wedge \psi_k(x,x_1,\ldots,x_n))$  has the required properties.
\end{proof}

\section{Amalgamation and joint consistency lemmas} \label{amal}

The results in this section are about amalgamation of Steiner quasigroups and applications to joint consistency questions of formulas in $T^\ast_\mathrm{Sq}$.   If $\overline{a},\overline{b}$ are (finite or infinite)  tuples of elements of $\c_\mathrm{Sq}$, the notation $\overline{a}\equiv_C \overline{b}$ is standard. Throughout this section,   if $A,B,C\subseteq \c_\mathrm{Sq}$ we use the notation $A\equiv_C B$ to mean that enumerations $\overline{a}$ of $A$ and $\overline{b}$ of $B$ have been fixed, and $\overline{a}\equiv_C\overline{b}$.  By elimination of quantifiers, this is equivalent to the existence of an isomorphism $(\langle A\cup C \rangle, \cdot)\cong (\langle B\cup C\rangle,\cdot)$  which is the identity on $C$ and maps $A$ onto $B$  respecting the enumerations $\overline{a}$ and $\overline{b}$.

For ease of notation, in this section we often use juxtaposition to denote unions of two or more sets.  As usual, $\P$ denotes the graph of the product in $(\c_\mathrm{Sq},\cdot)$.

The next proposition is not used in the rest of the paper, but it is included because the proof gives a flavour of the method used to prove the more complex statement of Proposition~\ref{al2.5} below.

\begin{prop}\label{al1} Let $A_0, A_1,B_0, B_1\subseteq \c_\mathrm{Sq}$  be  closed under product and suppose that  \[A_0\cap B_0= A_1\cap B_1= B_0 \cap B_1 = \emptyset, \mbox{  and   }  A_0 B_0\equiv A_1 B_1\,. \]
  Then there is a Steiner quasigroup $(A,\cdot)\subseteq (\c_\mathrm{Sq},\cdot)$ such that     $A\equiv_{B_0} A_0$ and $A\equiv_{B_1} A_1$.
\end{prop}
\begin{proof} Let $A,U$ and $V$ be  sets such that  
\begin{itemize}
\item $|A|=|A_0|=|A_1|$,  
\item $|U|= |\langle A_0 B_0\rangle \smallsetminus  A_0 B_0|$
\item   $|V|= |\langle  A_1 B_1\rangle \smallsetminus A_1 B_1|$,  
\item $U\cap A= U\cap V= V\cap A=\emptyset$  and  
\item $AUV \cap \langle A_0 A_1 B_0 B_1  \rangle = \emptyset$.   
\end{itemize}
We will define a partial  STS  on $ A U V \langle B_0 B_1\rangle$.  Let $f:\langle A_0 B_0\rangle \rightarrow \langle A_1 B_1 \rangle$  be an isomorphism of Steiner quasigroups that maps  $A_0$ onto $A_1$ and $B_0$ onto $B_1$. Fix a bijection \break $g:\langle A_0 B_0\rangle \rightarrow  A U B_0$  which is the identity on $B_0$ and maps $A_0$ onto  $A$, and a bijection $h:\langle A_1 B_1 \rangle \rightarrow A V B_1  $  which is the identity on $B_1$, and $h\restriction A_1 = g\circ f^{-1}\restriction A_1$.  
Let  $R$ be a relation on $A U B_0$  such that $g$ is an isomorphism between  $(\langle  A_0 B_0\rangle, \P^{ \langle  A_0 B_0\rangle})$  and $(AUB_0,R)$ and let  $S$ be a relation on $AVB_1$  such that $h$ is an isomorphism between  $(\langle A_1 B_1 \rangle, \, \P^{ \langle A_1 B_1 \rangle})$  and $(AVB_1, S)$.  Since  $(A U B_0 )\cap ( A V B_1)=A$  and $R^{ A}= S^{A}$,   it follows that $(A U  V B_0  B_1,\, R\cup S)$ is a partial STS. Since $\langle B_0 B_1\rangle \cap (A U V B_0 B_1)= B_0 B_1$  and  
\[(R\cup S)^{B_0 B_1} = \P^{ B_0}\cup \P^{B_1}= \P^{B_0 B_1}\, ,\]
 we have that 
  $( A U V \langle B_0 B_1\rangle, \P^{\langle B_0 B_1\rangle} \cup R\cup S)$  is a partial STS  that contains $(\langle B_0 B_1\rangle, \P^{ \langle B_0 B_1\rangle}~)$ as a substructure. By Remark~\ref{mc5}, there is an embedding 
  \[j : ( A U V \langle B_0 B_1\rangle, \P^{\langle B_0 B_1\rangle}\cup R\cup S) \rightarrow(\c_\mathrm{Sq}, \P)\]
   over $\langle B_0 B_1\rangle$.  Clearly,  $j(A)$  satisfies all our requirements.
\end{proof}

The next corollary shows the relevance of Proposition~\ref{al1} to joint consistency questions. It uses the notation in Definition~\ref{ac0} as well as the following.
\begin{nota}  For tuples $\overline{a}$ and $\overline{b}$, we write    $\overline{a}\equiv^k\overline{b}$  to mean that  $\overline{a}$ and $\overline{b}$ satisfy the same equalities between terms of rank $\leq k$.
\end{nota}

\begin{cor}\label{al2}  For every   formula  $\varphi(\overline{x},\overline{y})$ of the product language $L=\{\cdot\}$, there is a natural number $k$  such that, for any finite tuples  $\overline{a},\overline{b},\overline{c},\overline{d}$ of elements of $\c_\mathrm{Sq}$, if 
\[ \langle \overline{a}\rangle_k\cap \langle \overline{b}\rangle_k =\emptyset = \langle \overline{c}\rangle_k \cap \langle \overline{a}\rangle_k \ \mbox{ and } \ \overline{c},\overline{a}\equiv^k\overline{d},\overline{b}\ \mbox{ and }\ (\c_\mathrm{Sq},\cdot)\models \varphi(\overline{c},\overline{a})\, ,\] 
then $(\c_\mathrm{Sq},\cdot)\models \exists \overline{x}(\varphi(\overline{x},\overline{a})\wedge \varphi(\overline{x},\overline{b}))$.
\end{cor}
\begin{proof}   Let  $\Sigma(\overline{x},\overline{y})$  be the set of all inequalities  of the form $t(\overline{x})\neq t^\prime(\overline{y})$. Let $\overline{u},\overline{v}$ be tuples of variables having the same length as $\overline{x}$ and $\overline{y}$ respectively.  Let  $\Gamma(\overline{x},\overline{y}, \overline{u},\overline{v})$   be the set of all formulas of the form  
$t(\overline{x},\overline{y})=t^\prime(\overline{x},\overline{y})\leftrightarrow  t(\overline{u},\overline{v})=t^\prime(\overline{u},\overline{v})\,.$

  By Proposition~\ref{al1}, the following implication holds in $T^\ast_\mathrm{Sq}$:
$$\Sigma(\overline{y},\overline{v})\cup \Sigma(\overline{x},\overline{y})\cup \Gamma(\overline{x},\overline{y},\overline{u},\overline{v})\cup\{\varphi(\overline{x},\overline{y})\}\vdash \exists \overline{x}(\varphi(\overline{x},\overline{y})\wedge \varphi(\overline{x},\overline{v}))\, .$$
By compactness one gets finite subsets of $\Sigma$ and $\Gamma$ for which the same implication holds. The number $k$ is an upper bound for the ranks of the terms in these finite subsets.
\end{proof}

\begin{prop}\label{al2.5} Let $A_0,A_1,B_0,B_1\subseteq \c_\mathrm{Sq}$   be  closed under product and such that   \begin{itemize}
\item $A_0\cap B_0= A_1\cap B_1$ 
\item $E=B_0\cap B_1$
\item $A_0 B_0\equiv_E A_1 B_1$
\item $\langle A_0E\rangle \cap B_0 = \langle A_1E\rangle \cap B_1 =E$.
\end{itemize}
 Then there is a Steiner quasigroup $(A,\cdot)\subseteq (\c_\mathrm{Sq},\cdot)$ such that $A\equiv_{B_0} A_0$ and  $A\equiv_{B_1} A_1$.
\end{prop}
\begin{proof} 
Let $F=A_0\cap B_0= A_1\cap B_1$ and notice that $F\subseteq E$.  Now choose pairwise disjoint sets  $A,U,V,W$, each of which is also disjoint from $\c_\mathrm{Sq}$, and such that  
\begin{itemize}
\item $|A|=|A_0\smallsetminus F|=|A_1\smallsetminus F|$
\item $|W|= |\langle A_0E\rangle \smallsetminus A_0E| =  |\langle A_1E\rangle \smallsetminus A_1E|$
\item $|U|= |\langle A_0B_0\rangle \smallsetminus (\langle A_0E\rangle B_0)| = |V|= |\langle A_1B_1\rangle \smallsetminus (\langle A_1E\rangle B_1)|\,. $
\end{itemize}
 Let $f:\langle A_0B_0\rangle \rightarrow \langle A_1B_1\rangle$ be an isomorphism that is the identity on $E$, maps $A_0$ onto $A_1$ and maps $B_0$ onto $B_1$. Fix  bijections  $g:\langle A_0B_0\rangle \rightarrow  B_0AWU$   and    $h:\langle A_1 B_1\rangle \rightarrow  B_1AWV$  such that  
 \begin{itemize}
 \item $g$ is the identity on $B_0$ and $h$ is the identity on $B_1$ 
 \item $g(A_0\smallsetminus F)= A$ and $h(A_1\smallsetminus F)= A$
 \item $g(\langle A_0E\rangle\smallsetminus A_0 E)= W$ and $h(\langle A_1E\rangle\smallsetminus A_1 E)= W$  
 \item $g(\langle A_0B_0\rangle\smallsetminus \langle A_0 E\rangle  B_0)= U$ and  $h(\langle A_1B_1\rangle\smallsetminus \langle A_1 E\rangle B_1)= V$.
 \end{itemize}
We additionally require that $h\restriction\langle A_1E\rangle= g\circ f^{-1}\restriction \langle A_1E\rangle$.

Let $R$ be a ternary relation on $AWUB_0$  such that  $g$ is an isomorphism between $(\langle A_0B_0\rangle, \P^{\langle A_0B_0\rangle})$  and  $(AWUB_0, R)$. Similarly, let $S$ be  a ternary relation on $AWVB_1$  such that  $h$ is an isomorphism between $(\langle A_1B_1\rangle, \P^{\langle A_1B_1\rangle})$  and  $(AWVB_1, R)$.  We will show that $(AWUB_0, R)$  and  $(AWVB_1,S)$  are compatible and that  $(\langle B_0B_1\rangle, \P^{\langle B_0B_1\rangle})$  is compatible with both of them.

\noindent
\emph{Claim 1}. $R^{AEW}= S^{AEW}$.

\noindent
\emph{Proof of Claim 1}.   This is due to the fact that  $h\restriction \langle A_1E\rangle= g\circ f^{-1}\restriction \langle A_1E\rangle$.

\noindent
\emph{Claim 2}. $(AWUVB_0B_1, R\cup S)$  is a partial STS.

\noindent
\emph{Proof of Claim 2}.  Note that  $(AWUB_0)\cap(AWVB_1)= AEW$.  Let $a,b\in AEW$  and assume there is some $c\in AWUB_0$  such that  $R(a,b,c)$. We will show that  $S(a,b,c)$. By claim 1, it is enough to prove that $c\in AEW$. This is clear, since $g(\langle A_0E\rangle) =AEW$.

\noindent
\emph{Claim 3}. $(AWU \langle B_0B_1\rangle, \P^{\langle B_0B_1\rangle}\cup R)$   and   $(AWV\langle B_0B_1\rangle , \P^{\langle B_0B_1\rangle}\cup S)$  are partial STSs.

\noindent
\emph{Proof of Claim 3}. The first statement follows from the fact that $( AWU\langle B_0B_1\rangle, \P^{\langle B_0B_1\rangle}\cup R)$ is the union of the two partial STSs $(\langle B_0B_1\rangle , \P^{\langle B_0B_1\rangle})$ and $(AWUB_0,R)$,  whose intersection $B_0= \langle B_0B_1\rangle  \cap (AWUB_0)$ is relatively closed in both systems, and  $R^{B_0}= \P^{B_0}$. The second statement is similar.

By Lemma~\ref{mc2}  and claims 2 and 3,  $( AWUV\langle B_0B_1\rangle, \P^{\langle B_0B_1\rangle}\cup R\cup S)$  is a partial STS, and it clearly contains  $(\langle B_0B_1\rangle, \P^{\langle B_0B_1\rangle})$ as a substructure.  By Remark~\ref{mc5}, there is an embedding $j$ of  $( AWUV\langle B_0B_1\rangle, \P^{\langle B_0B_1\rangle}\cup R\cup S)$ in $\c_\mathrm{Sq}$  over $\langle B_0B_1\rangle$.  Clearly,   
\[j(AF)\cong_{B_0}AF= g(A_0)\cong_{B_0} A_0 \mbox{ and }j(AF)\cong_{B_1}AF= h(A_1)\cong_{B_1} A_1\, .\]
\end{proof}

\begin{prop}\label{al3} Let $A_0,A_1,B_0,B_1,D\subseteq \c_\mathrm{Sq}$   be  closed under product and such that   
\begin{itemize}
\item $A_0\cap B_0= A_1\cap B_1, \ E=B_0\cap B_1$ and     
\item $A_0 B_0\equiv_E A_1 B_1, \ D\equiv_{EA_0} B_0 \mbox{ and }  D\equiv_{EA_1} B_1\, .$
\end{itemize}
   Then there is a Steiner quasigroup $(A,\cdot)\subseteq (\c_\mathrm{Sq},\cdot)$ such that $A\equiv_{B_0} A_0$ and  $A\equiv_{B_1} A_1$.
\end{prop}
\begin{proof}  We check that  $\langle A_0E\rangle \cap B_0 =E$ and $\langle A_1E\rangle \cap B_1 =E$. Then Proposition~\ref{al2.5} applies. 
 It is enough to check the first equality. Assume  $a\in \langle A_0E\rangle \cap B_0$. There are terms $t(\overline{x},\overline{y})$, $r(\overline{z})$  and finite tuples $\overline{a}_0\in A_0$, $\overline{e}\in E$  and  $\overline{b}_0\in B_0$  such that  $a= t(\overline{a}_0,\overline{e})= r(\overline{b}_0)$.  By the assumptions on $D$, there is a finite tuple $\overline{d}\in D$
such that $\overline{d} \equiv_{EA_0} \overline{b}_0$, so that  $t(a_0,\overline{e})=r(\overline{d})$. Now let $\overline{b}_1 \in B_1$ be such that $\overline{d} \equiv_{EA_1} \overline{b}_1$. Then $\overline{b}_1 \equiv_E \overline b_0$, and so there is $\overline{a}_1 \in A_1$ such that $\overline{a}_0\overline{b}_0 \equiv_E \overline{a}_1 \overline{b}_1$. Hence $t(\overline{a}_1,\overline{e})= r(\overline{b}_1)$. Since $\overline{d} \equiv_{EA_1} \overline{b}_1$, we have that $t(\overline{a}_1,\overline{e})= r(\overline{d})$. It follows that $a= r(\overline{b}_1)\in B_1$  and, therefore,  $a\in B_0\cap B_1 =E$.
\end{proof}

Proposition~\ref{al2.5} has been stated and proved for two Steiner triple systems $(A_0,B_0)$ and $(A_1,B_1)$, but in fact the result holds for arbitrary families $\{(A_i, B_i) \mid i<\omega\}$ of Steiner triple systems. We omit the proof, since it is a straightforward adaptation of that of Proposition~\ref{al2.5}.

\begin{remark}\label{al4}    Let $\{A_i\mid i\in I\}$ and $\{B_i\mid i\in I\}$ be families of subsets of  $\c_\mathrm{Sq}$    closed under product and such that   $A_i\cap B_i= A_j\cap B_j$, $E=B_i\cap B_j$ (if $i\neq j$),   $A_i B_i\equiv_E A_j B_j$  and     $\langle A_i E\rangle \cap B_i =E$.    Then there is some Steiner quasigroup $(A,\cdot)\subseteq (\c_\mathrm{Sq},\cdot)$ such that $A\equiv_{B_i} A_i$  for every $i\in I$.
\end{remark}

\section{\tpt and \nsop} \label{tp2andnsop1}

Recall that a formula $\varphi(\overline{x};\overline{y})$  has the \emph{tree property of the second kind} (\tpt)  in $T$ if in the monster model of $T$ there is an array of tuples $(\overline{a}_{ij}\mid i,j<\omega)$ and some natural number $k$ such  that for each $i < \omega$ the set $\{\varphi(\overline{x},\overline{a}_{ij})\mid  j<\omega\}$  is $k$-inconsistent, and for each $f:\omega\rightarrow \omega$ the path $\{\varphi(\overline{x},\overline{a}_{if(i)})\mid i<\omega \}$  is consistent.  We say that $T$ is \tpt if some formula has \tpt in $T$. Otherwise $T$ is \ntpt.

Also recall that the formula $\varphi(\overline{x},\overline{y})$ has the \emph{1-strong order property}, \sop, if  there is a tree of tuples of parameters $(\overline{a}_s\mid s\in 2^{<\omega})$  such that for every $f:\omega\rightarrow 2$,  the branch  $\{\varphi(\overline{x},\overline{a}_{f\restriction n})\mid n<\omega\}$  is consistent and for every $s,t\in 2^{<\omega}$  with  $s^\smallfrown 0\subseteq t$,   $\varphi(\overline{x},\overline{a}_t)\wedge \varphi(\overline{x},\overline{a}_{s^\smallfrown 1})$ is inconsistent. The theory $T$ is \sop  if some formula has \sop in $T$. Otherwise, it is \nsop.

\tpt and \sop, as well as their negations \ntpt and \nsop, are dividing lines in the classification of first-order theories. They were first introduced by Shelah in~\cite{Sh80a}.  \ntpt theories include simple and \nip theories, and have received a lot of attention recently 
-- see~\cite{Chernikov12} and ~\cite{ChernikovKaplan09}. \nsop theories are the first level in the $\mathrm{NSOP}_n\,$ hierarchy, a family of theories without the strict order property that properly extends the class of simple theories.  It is  not known whether \nsop and $\mathrm{NSOP}_2\,$ are equivalent. $\mathrm{NSOP}_2\,$ is equivalent to $\mathrm{NTP}_1\,$, the negation of the tree property of the first kind. Shelah proved that a theory is simple if and only if it is \ntpt and $\mathrm{NTP}_1\,$.  The class of \nsop  theories has recently become the object of close scrutiny and new natural examples are being discovered -- see ~\cite{ChernikovRamsey16}, \cite{KaplanRamsey17} and~\cite{KaplanRamseyShelah17}.
In this section we show that $T^\ast_\mathrm{Sq}$ is \tpt and \nsop, thus adding a further example of a \tpt and \nsop theory to those described in~\cite{ConantKruckman17}. 

\begin{remark}\label{tpt1}  In  any Steiner quasigroup, the following cancellation law holds:  
\[\forall xyz\,( x\cdot y= x\cdot z\rightarrow y=z)\, .\]
This is because if $x\cdot y= x\cdot z$,  then $y= x\cdot (x\cdot y) = x\cdot (x\cdot z)= z$.
\end{remark}

\begin{prop}\label{tpt2} The formula  $\varphi(x;y_1,y_2,y_3)\equiv x= (y_1\cdot (y_2\cdot (y_3\cdot x)))$   has \tpt  in  $T^\ast_\mathrm{Sq}$. Hence $T^\ast_\mathrm{Sq}$  is \tpt.
\end{prop}
\begin{proof}  We embed in $(\c_\mathrm{Sq}, \P)$ a partial STS  which 
contains an array  $(a_ib_ic_{ij}\mid i,j<\omega)$  and  a sequence  $(d_f\mid f\in \omega^\omega)$  
such that for each $i \in \omega$ the set $\{\varphi(x;a_i,b_i,c_{ij})\mid j<\omega\}$  is   $2$-inconsistent, and each $d_f$  realizes the corresponding path $\{\varphi(x;a_i,b_i,c_{if(i)})\mid i<\omega\}$. The array will be chosen in such a way that $c_{ij}\neq c_{ik}$  for all  $j\neq k$, so that Remark~\ref{tpt1}  implies the inconsistency of  \[\varphi(x,a_i,b_i,c_{ij})\wedge \varphi(x,a_i,b_i,c_{ik})\,.\]  
We construct a suitable partial STS outside $\c_\mathrm{Sq}$. Remark~\ref{mc5} then gives the required embedding. Given $i,j<\omega$,   we choose a set  
\[A_{ij}=\{a_i,b_i,c_{ij}\}\cup \{d_f\mid f\in\omega^\omega, f(i)=j\}\cup \{a_{ijf}^\ast,b_{ijf}^\ast\mid  f\in\omega^\omega, f(i)=j\}\, \]
of elements not in $\c_\mathrm{Sq}$. It is understood that for all $i$, $j$ and $f$ the elements $a_i, b_i, c_{ij}$ and $d_f$ are pairwise distinct, and therefore  $A_{ij}\cap A_{ik}= \{a_i,b_i\} $  if  $j\neq k$    and   $A_{ij}\cap A_{lk}= \{d_f\mid  f(i)=j  \mbox{ and } f(l)=k\}$  if  $i\neq l$. Now we define a partial STS $(A_{ij},R_{ij})$ on each  set $A_{ij}$.  The relation $R_{ij}$  will contain the triples $(d_f,c_{ij}, b^\ast_{ijf})$, $(b^\ast_{ijf},b_i,a^\ast_{ijf})$,  $(d_f,a_i,a^\ast_{ijf})$ and all their permutations, as well as all the triples of the form $(a,a,a)$  with  $a\in A_{ij}$. 
 It is easy to check that no product is doubly defined.  Observe that this choice of $R_{ij}$  gives, in product notation,
   \[d_f= a_i\cdot a^\ast_{ijf}=a_i\cdot (b_i\cdot b^\ast_{ijf})= a_i\cdot (b_i\cdot (c_{ij}\cdot d_f))\, ,\]  
 and therefore for all $i, j < \omega$ we have that  $(d_f;a_i,b_i,c_{ij})$  satisfy  $\varphi(x;y_1,y_2,y_3)$.  Now, if we take $j\neq k$, then the two elements $a_i,b_i$ of the intersection $A_{ij}\cap A_ {ik}$ do not have a defined product either in $(A_{ij},R_{ij})$  or in $(A_{ik},R_{ik})$.   Hence,  $(A_{ij}\cup A_{ik}, R_{ij}\cup R_{ik})$  is a partial STS.  Let  $A_i=\bigcup_{j<\omega}A_{ij}$  and  $R_i=\bigcup_{j<\omega}R_{ij}$. By Lemma~\ref{mc2},  each  $(A_i,R_i)$ is a partial STS.  Now 
 let $i,l<\omega$ be different. Then $A_i\cap A_l= \{d_f\mid f\in \omega^\omega\}$,  and for $f\neq g$ the product of $d_f$ and $d_g$ is not defined either in $(A_i,R_i)$ or in $(A_l,R_l)$. Hence $(A_i\cup A_l,R_i\cup R_l)$ is a partial STS. Finally, let $A= \bigcup_{i<\omega}A_i$  and $R=\bigcup_{i<\omega}R_i$. Again by Lemma~\ref{mc2}, we have that $(A,R)$  is a partial STS. By Remark~\ref{mc5}, there is an embedding $h:(A,R)\rightarrow (\c_\mathrm{Sq},\P)$ and so for each $i,j<\omega$ and for each $f\in \omega^\omega$ such that $f(i)=j$,  
 \[h(d_f)=h(a_i)\cdot (h(b_i)\cdot(h(c_{ij})\cdot h(d_f)))\]
  and therefore  $(\c_\mathrm{Sq},\cdot)\models \varphi(h(d_f);h(a_i),h(b_i),h(c_{ij}))$.
\end{proof}

\begin{prop}\label{tpt3} In $T^\ast_\mathrm{Sq}$, nonalgebraic formulas of the form $\varphi(x;b,c)$  do not divide over $\emptyset$.
\end{prop}
\begin{proof} Assume for a contradiction that $\varphi(x;b,c)$  divides over $\emptyset$, and let $(b_ic_i\mid i<\omega)$,  where  $b_0c_0=bc$, be an indiscernible sequence that witnesses the dividing, so that $\{\varphi(x;b_i,c_i)\mid i<\omega\}$ is inconsistent.  We will use Remark~\ref{al4} to contradict the inconsistency of this set.   Choose   $(a_i\mid i<\omega)$  such that  $\models \varphi(a_0;b_0,c_0)$  and  $a_ib_ic_i\equiv a_jb_jc_j$  for all $i,j<\omega$.  Since $\varphi(x;b,c)$ is not algebraic, we may assume that $a_i\not\in \langle b_i,c_i\rangle$. There are several cases to be considered. Let us consider first the case $b_0=c_0$. This implies $b_i= c_i$ for all $i<\omega$.  Notice that $b_i\neq b_j$  if $i\neq j$.  By Remark~\ref{al4},  with  $A_i=\langle a_i\rangle = \{a_i\}$, $B_i =\langle b_i,c_i\rangle = \{b_i\}$  and $E=\emptyset$, there is $a$  such that  $a\equiv_{b_i}a_i$  for every $i<\omega$. Then $a$ realizes each $\varphi(x;b_i,c_i)$.   Now assume that $b_0 \neq c_0$, so that for all $i$ we have  $b_i\neq c_i$.  If  $\langle b_i,c_i\rangle$  and $\langle b_j,c_j\rangle$  (with $i\neq j$)  share two elements, then they are equal  and  we get $b_i=b_j$  and $c_i=c_j$  for all $i,j$.  Assume $\langle b_i,c_i\rangle$  and $\langle b_j,c_j\rangle$  (with $i\neq j$)  share one element $e$. Then without loss of generality $b_i\cdot c_i = e$ for all $i$. We apply again Remark~\ref{al4}, with $A_i=\langle a_i\rangle = \{a_i\}$, $B_i =\langle b_i,c_i\rangle = \{b_i,c_i,e\}$  and $E=\{e\}$. Notice that $\langle a_i,e\rangle =\{a_i,e,a_i\cdot e\}$  and $a_i\cdot e\neq b_i,c_i$.  The case where  $\langle b_i,c_i\rangle\cap\langle b_j,c_j\rangle=\emptyset$  (with $i\neq j$)  is similar, with $E=\emptyset$.
\end{proof}

The next corollary shows that the formula in Proposition~\ref{tpt2} is optimal, in the sense that no formula of the form $\varphi(x;\overline{y})$, where $\overline{y}$ has fewer than three variables, is \tpt.
\begin{cor}\label{tpt4} In $T^\ast_\mathrm{Sq}$, no formula of the form $\varphi(x;y_1,y_2)$ has \tpt.
\end{cor}
\begin{proof} Let $\{\varphi(x;b_{ij}, c_{ij})\mid i,j<\omega\}$ be an array of formulas that witnesses \tpt, where the tuple $\overline{x}$ in our definition of \tpt is the single variable $x$ and the tuple of parameters $\overline{a}_{ij}$ consists of the two elements $b_{ij}c_{ij}$. We can find such an array with the property that $b_{ij}c_{ij}\equiv b_{ik}c_{ik}$ for all $i,j,k<\omega$. Consider the first row $\{\varphi(x;b_{0j}, c_{0j})\mid j<\omega\}$. Each $\varphi(x;b_{0j},c_{0j})$ is nonalgebraic. Since $\{\varphi(x;b_{0j},c_{0j})\mid j<\omega\}$ is $k$-inconsistent for some $k$, the row witnesses that $\varphi(x;b_{00},c_{00})$ $k$-divides over $\emptyset$. But this contradicts Proposition~\ref{tpt3}.
\end{proof}

The notation  $\overline{b}_0\ind^u_M \overline{b}_1$  used in Fact~\ref{sop1} below  means that $\tp(\overline{b}_0/M\overline{b}_1)$ is a coheir of  $\tp(\overline{b}_0/M)$, and satisfaction of formulas is meant in the monster model of the theory.

\begin{fact}\label{sop1} Assume $\varphi(\overline{x},\overline{y})$  witnesses \sop. Then there are $M$, $\overline{a}_0,\overline{a}_1,\overline{b}_0,\overline{b}_1$  so that  $\overline{b}_0\ind^u_M \overline{b}_1$, $\overline{b}_0\ind^u_M \overline{a}_0$, 
$\overline{a}_0\overline{b}_0\equiv_M \overline{a}_1\overline{b}_1$ and $\models \varphi(\overline{a}_0,\overline{b}_0)\wedge \varphi(\overline{a}_1,\overline{b}_1)$  but  $\varphi(\overline{x},\overline{b}_0)\wedge \varphi(\overline{x},\overline{b}_1)$ is inconsistent.
\end{fact}
\begin{proof} Proposition 5.2 in~\cite{ChernikovRamsey16}.
\end{proof}

\begin{prop}\label{sop2} $T^\ast_\mathrm{Sq}$  is  \nsop.
\end{prop}
\begin{proof}    Assume $\varphi(\overline{x};\overline{y})$ witnesses \sop. By Fact~\ref{sop1}, there are tuples $\overline{a}_0,\overline{b}_0,\overline{a}_1,\overline{b}_1$ and a model $M$  such that  
\begin{itemize}
\item $\overline{a}_0\overline{b}_0\equiv_M \overline{a}_1\overline{b}_1$, 
\item $(\c_\mathrm{Sq},\cdot)\models \varphi(\overline{a}_0,\overline{b}_0)\wedge \varphi(\overline{a}_1,\overline{b}_1)$,  
\item $\varphi(\overline{x},\overline{b}_0)\wedge \varphi(\overline{x},\overline{b}_1)$ is inconsistent, and 
\item the types $\tp(\overline{b}_0/M\overline{b}_1)$ and $\tp(\overline{b}_0/M\overline{a}_0)$ 
are coheirs of their restriction to $M$. 
\end{itemize}
Our goal is to show that we can realize $\varphi(\overline{x},\overline{b}_0)\wedge \varphi(\overline{x},\overline{b}_1)$.   Nothing changes if one replaces each tuple by a tuple enumerating the substructure generated by it  and we will assume that the replacement has been made.

  Now  let $\overline{e}$ enumerate $\overline{b}_0\cap \overline{b}_1$.   Since $\tp(\overline{b}_0/M\overline{b}_1)$ does not fork over $M$, $\overline{e}$ is a tuple of $M$.  We claim that there is a tuple $\overline{d}$  such that  $\overline{d}\equiv_{\overline{e}\overline{a}_0} \overline{b}_0$   and   $\overline{d}\equiv_{\overline{e}\overline{a}_1} \overline{b}_1$.  Let $p(\overline{x},\overline{y})=\tp(\overline{a}_0\overline{b}_0/\overline{e})$. Note that $p(\overline{x},\overline{y})=\tp(\overline{a}_1\overline{b}_1/\overline{e})$.  We want to check the consistency of $p(\overline{a}_0,\overline{y})\cup p(\overline{a}_1,\overline{y})$.  Let  $\psi(\overline{x},\overline{y})\in p(\overline{x},\overline{y})$. 
   Since  $\tp(\overline{b}_0/M\overline{a}_0)$ is a coheir of its restriction to $M$, there  is some tuple $\overline{m}\in M$  such that  $\models \psi(\overline{a}_0,\overline{m})$.  Since $\overline{a}_0\equiv_M \overline{a}_1$, $\models \psi(\overline{a}_1,\overline{m})$ and, therefore, $\psi(\overline{a}_0,\overline{y})\wedge \psi(\overline{a}_1,\overline{y})$ is consistent.
  
 Finally, note that the coheir assumptions imply additionally that $\overline{a}_0\cap\overline{b}_0$  and  $\overline{a}_1\cap\overline{b}_1$  are contained in $M$, and hence they coincide.
  Then Proposition~\ref{al3} gives a tuple $\overline{c}$  such that   $\overline{c}\equiv_{\overline{b}_0}a_0$   and  $\overline{c}\equiv_{\overline{b}_1}a_1$.  But then   $\models \varphi(\overline{c};\overline{b}_0)\wedge \varphi(\overline{c};\overline{b}_1)$.
  
\end{proof}

\section{Hyperimaginaries and imaginaries} \label{section8}

In this section we prove that $T^\ast_\mathrm{Sq}$ has elimination of hyperimaginaries and weak elimination of imaginaries. We use the method due to Conant and described in~\cite{Conant17}. We first discuss briefly its main ideas.

Consider an arbitrary complete theory $T$. Let $\overline{a}$ be a tuple in the monster model of $T$, possibly infinite, and let $E$ be an equivalence relation between tuples of the same length as $\overline{a}$. Assume that $E$ is type-definable over the empty set. Then $\overline{a}_E$ is a hyperimaginary; if $\overline{a}$ is finite and $E$ is definable, it is an imaginary.  It is well known that if there is a (possibly infinite) tuple $\overline{b}$  such that   $\overline{a}_E \in \dcl(\overline{b})$  and  $\overline{b}\in\bdd(\overline{a}_E)$, then $\overline{a}_E$ is eliminable (see, for instance, Lemma 18.6 in~\cite{Casanovas11}).  When $\ov{a}_E$ is an imaginary, the tuple $\ov{b}$ can be chosen to be finite and  
in $\acl(\ov{a}_E)$. If for every hyperimaginary $\ov{a}_E$ such a tuple $\ov{b}$ can be found, then $T$ has elimination of hyperimaginaries and weak elimination of imaginaries. 

\begin{defi} Let $\overline{a}_E$ be as above. Then $\Sigma(\ov{a},E)$ is the set of all the subtuples $\ov{c}$ of $\ov{a}$ for which there is an indiscernible sequence $(\ov{a}_i\mid i<\omega)$  with $\ov{a}=\ov{a}_0$ and $E(\ov{a}_i,\ov{a}_j)$ for all $i,j$, and such that $\ov{c}$ is the common intersection of all the $\ov{a}_i$.  The set  $\Sigma(\ov{a},E)$ is partially ordered by the relation of being a subtuple. 
\end{defi}

\begin{fact}\label{im1} Let $\overline{a}_E$ be a hyperimaginary.
\begin{enumerate}
\item There are minimal elements in $\Sigma(\ov{a},E)$.
\item If $\ov{b}$ is a minimal element of $\Sigma(\ov{a},E)$, then $\ov{b}\in \bdd(\ov{a}_E)$.
\item Assume  $\ov{a}$ enumerates an algebraically closed set and there is a ternary relation $\ind$ between subsets of the monster model of $T$ with the following properties:
\begin{enumerate}
\item Invariance: if $A\ind_C B$ and $f$ is an automorphism of the monster model, then $f(A)\ind_{f(C)}f(B)$.
\item Monotonicity: if $A\ind_C B$, then  $A_0\ind_C B_0$  for every $A_0\subseteq A$ and $B_0\subseteq B$.
\item Full existence over algebraically closed sets:  for all $A,B,C$, if $C$ is algebraically closed, then $A^\prime\ind_C B$  for some $A^\prime\equiv_C A$.
\item Stationarity: if $A,A^\prime,B,C$ are algebraically closed sets such that  $A\equiv_C A^\prime$, $C\subseteq A\cap B$,  $A\ind_C B$ and $A^\prime\ind_C B$, then $A\equiv_B A^\prime$.
\item Freedom:  for all $A,B,C$, if  $A\ind_C B$ and  $C\cap (AB)\subseteq D\subseteq C$, then   $A\ind_D B$.
\end{enumerate}
Then $\ov{a}_E\in \dcl(\ov{b})$  for every $\ov{b}\in \Sigma(\ov{a},E)$.
\end{enumerate}
\end{fact}
\begin{proof} By lemmas 5.2, 5.4 and 5.5 of~\cite{Conant17}.
\end{proof}

We will  apply a minor modification of Fact~\ref{im1} to our theory $T^\ast_\mathrm{Sq}$. For this, we need to define a suitable relation $\ind$.

\begin{defi} Let $(A,R)$ be a partial STS and let $(B,\cdot)$ a Steiner quasigroup. Let $(B,S)$ be the STS corresponding to $(B, \cdot)$. We say that $f:A\rightarrow B$ is a \textbf{homomorphism}  if it is a homomorphism between the relational structures $(A,R)$ and $(B,S)$. Equivalently, $f:A \rightarrow B$ is a homomorphism if, for all $a,b,c\in A$, 
\[ R(a,b,c) \ \Rightarrow f(a)\cdot f(b)= f(c)\, .\]
 A Steiner quasigroup $(B,\cdot)$  is \textbf{freely generated} by $A\subseteq B$ if $\langle A\rangle = B$ and every homomorphism from $(A,R)$  (where $R$ is the restriction to $A$ of the graph of the product on $B$) into some Steiner quasigroup $(C,\cdot)$ can be extended to a homomorphism of $(B,\cdot)$ into $(C,\cdot)$.
\end{defi}

\begin{remark}\label{im2} Every   partial STS $(A,R)$  can be extended to some STS $(B,S)$ whose corresponding Steiner quasigroup $(B,\cdot)$ is freely generated by $A$. Moreover,  $(B,\cdot)$ is unique up to isomorphism over $A$.
\end{remark}
\begin{proof} For $a \in A$, let $c_a$ be a constant symbol and extend the language $L=\{\cdot\}$ to $L(A)=L\cup\{c_a\mid a\in A\}$. Let $K$ be the class of all $L(A)$-structures $(M,\cdot, c_a^M)_{a\in A}$  such that $(M,\cdot)$ is a Steiner quasigroup and the mapping $a\mapsto c_a^M$ defines a homomorphism of $(A,R)$ into $(M,\cdot)$. Since $K$ is closed under substructures, direct products and homomorphic images, it is a variety (in the sense of universal algebra) and therefore it contains a free algebra $(F,\cdot, c_a^F)_{a\in A}$, which is unique up to isomorphism.  As an $L(A)$-structure, $(F,\cdot, c_a^F)_{a\in A}$ is freely generated 
by the empty set.  Since there are structures in $K$ whose corresponding STS restricted to $A$ is $(A,R)$,  the mapping $a\mapsto c_a^F$ defines an isomorphism of partial STSs and we can assume that $a=c_a^F$ and $(A,R)$ is a substructure of the STS associated to $(F,\cdot)$ and $F=\langle A\rangle$. It is easy to see that $(F,\cdot)$ satisfies our requirements.
\end{proof}

\begin{remark}\label{im3} If $(B,\cdot)$ is a Steiner quasigroup, $S$ is the graph of the product on $B$  and $A\subseteq B$, the following are equivalent:
\begin{enumerate}
\item $(B,\cdot)$ is freely generated by $A$;
\item   $(B,S) =\bigcup_{n<\omega}(A_n,R_n)$  for some chain of partial STSs $(A_n,R_n)$  such that:
\begin{enumerate}
\item  $A_0 =A$
\item  $A_{n+1}= \{a\cdot b\mid a,b\in A_n\}$
\item For every $c\in A_{n+1}\smallsetminus A_n$  there is a unique pair $\{a,b\}\subseteq A_n$  such that  $a\cdot b= c$
\item $R_n= S^{A_n}$.
\end{enumerate}
\end{enumerate}
\end{remark}
\begin{proof} Assume $(B,\cdot)$ is as in 2, let $(C,\cdot)$ be a Steiner quasigroup and let $f:A\rightarrow C$ be a homomorphism of $(A,S^A)$ into $(C,\cdot)$. Using the uniqueness condition (c), we can inductively define an ascending chain of homomorphisms $f_n:A_n\rightarrow C$  of  $(A_n,R_n)$ into $(C,\cdot)$ starting with $f_0=f$. Then $\bigcup_{n<\omega} f_n$ is a homomorphism from $(B,S)$ into $(C,\cdot)$ that extends~$f$.

The other direction follows from this and the uniqueness of the freely generated structure.
\end{proof}

\begin{defi} For subsets $A,B,C$ of the monster model $(\c_\mathrm{Sq},\cdot)$ of $T^\ast_\mathrm{Sq}$, define  $A\ind_C B$  if and only if  $\langle AC\rangle \cap \langle BC\rangle =\langle C\rangle$  and  $\langle ABC\rangle$ is freely generated by $\langle AC\rangle \langle BC\rangle$.
\end{defi}

\begin{remark}\label{im4} It is clear that $\ind$ is invariant and  symmetric. Moreover,    $A\ind_C B$ if and only if $\langle AC\rangle \ind_{\langle C\rangle}\langle BC\rangle$.
\end{remark}

In the next lemmas we check that the ternary relation $\ind$ satisfies the remaining properties in Fact~\ref{im1}, with the exception of freedom.  Instead, in Lemma~\ref{im8} we prove  a weak version of freedom which suffices for elimination of hyperimaginaries in our setting and seems more appropriate for a language with function symbols. Recall that in the monster model $(\c_\mathrm{Sq},\cdot)$ of  $T^\ast_\mathrm{Sq}$  we have $\acl(A)=\langle A\rangle$.  Also recall that $\P$ is the graph of the product in $\c_\mathrm{Sq}$. We will say that a set $A$ is \emph{closed} if $A=\langle A\rangle$. In the rest of this section we use elimination of quantifiers for $T^\ast_\mathrm{Sq}$ and Remark~\ref{mc5} without explicit mention.

\begin{lemma}\label{extension} Assume $(A,R)\subseteq (A^\prime,R^\prime)$ and  $(B,S)\subseteq (B^\prime,S^\prime)$ are STSs and $(AB, R\cup S)$ and $(A^\prime B^\prime,R^\prime\cup S^\prime)$ are partial STSs.  Assume additionally that $A\cap B^\prime = A\cap B = A^\prime\cap B$. If $f$ is a homomorphism from the partial STS  $(AB, R\cup S)$ into a Steiner quasigroup $(E,\cdot)$, then there is some Steiner quasigroup $(E^\prime,\cdot)$ extending $(E,\cdot)$ and some homomorphism $f^\prime \supseteq f$ from the partial STS $(A^\prime B^\prime,R^\prime\cup S^\prime)$ into $(E^\prime,\cdot)$.
\end{lemma}
\begin{proof}  Choose pairwise disjoint sets $U,V,W$  that are disjoint from $E$ and such that $|W|=|(A^\prime\cap B^\prime) \smallsetminus (A\cap B)|$, $|U|= |A^\prime\smallsetminus AB^\prime|$  and   $|V|= |B^\prime\smallsetminus A^\prime B|$, choose a bijection $h:  (A^\prime\cap B^\prime)\smallsetminus (A\cap B)\rightarrow W$  and extend $h$ to bijections $h_1: A^\prime\smallsetminus A\rightarrow UW$ and $h_2:B^\prime\smallsetminus B\rightarrow VW$. Let $P$ be the graph of the product on $E$. Define a ternary relation $R_f$ on  $f(A)UW$ by adding to  $P^{f(A)}\cup h_1(R^\prime\restriction (A^\prime\smallsetminus A))$  all  triples of the form  $(h_1(a),h_1(b),f(a\cdot b))$  with $a,b\in A^\prime\smallsetminus A$ and  $a\cdot b\in A$  as well as its permutations and all triples of the form $(a,a,a)$. Similarly, define $S_f$ on  $f(B)VW$ by adding to $P^{f(B)}\cup h_2(S^\prime\restriction (B^\prime\smallsetminus B))$  all triples of the form $(h_2(a),h_2(b),f(a\cdot b))$  where $a,b\in B^\prime\smallsetminus B$ and $a\cdot b\in B$ and their  permutations, as well as identities $(a,a,a)$.  Now,  $(f(A)UW,R_f)$ and $(f(B) VW,S_f)$ are STSs and  
$$ (f\restriction A)\cup  h_1 : (A^\prime, R^\prime) \rightarrow (f(A)UW,R_f) \ \mbox{ and } (f\restriction B) \cup  h_2 : (B^\prime, R^\prime) \rightarrow (f(B) VW,S_f)\ $$
are homomorphisms. The STSs $(f(A)UW,R_f)$ and $(f(B) VW,S_f)$ are compatible on their intersection $(f(A) \cap f(B))W$ and hence $(f(AB)UVW, R_f\cup S_f)$ is a partial STS, and  moreover  
$$f^\prime =f\cup h_1\cup h_2: A^\prime B^\prime \rightarrow f(AB)UVW$$
 is a homomorphism.  Extend the partial STS  $(EUVW,  P\cup R_f\cup S_f)$  to an  STS, and let  $(E^\prime,\cdot)$ be its associated Steiner quasigroup. Then $(E,\cdot)\subseteq (E^\prime,\cdot)$  and  $f^\prime: A^\prime B^\prime \rightarrow E^\prime$ is a homomorphism extending $f$.
\end{proof}

\begin{lemma}\label{im5} $\ind$ is monotone.
\end{lemma}
\begin{proof} Since $\ind$ is symmetric,   it is enough to prove that $A\ind_C B$  implies $A\ind_C B_0$  for every $B_0\subseteq B$.  Moreover we may assume that $C\subseteq A\cap B_0$ and  $A,B,C,B_0$ are closed.  It is clear that $\langle AC\rangle \cap \langle B_0 C\rangle =\langle C\rangle$.  We check that $\langle AB_0\rangle$ is freely generated from $AB_0$. Let $R= \P^{AB_0}$ and let $f: AB_0\rightarrow D$ be a homomorphism of $(AB_0,R)$ to a Steiner quasigroup $(D,\cdot)$,  which  can be assumed to be  a substructure of $(\c_\mathrm{Sq},\cdot)$. We want to extend $f$ to some homomorphism from $(\langle AB_0\rangle,\cdot)$  to $(D,\cdot)$.  By Lemma~\ref{extension}, there is some Steiner quasigroup $(D^\prime,\cdot)\supseteq (D,\cdot)$ and some homomorphism
$f^\prime\supseteq f$ from the partial STS $(AB, \P^{AB})$ into $(D^\prime,\cdot)$.   Since we are assuming that $\langle AB\rangle$ is freely generated from $AB$, $f^\prime$ extends to some homomorphism $g$ from $(\langle AB\rangle, \cdot)$  into  $(D^\prime,\cdot)$. But $g(\langle AB_0\rangle)\subseteq D$ and so $g\restriction \langle A B_0\rangle$ is a homomorphism from $(\langle AB_0\rangle,\cdot)$ to $(D,\cdot)$  extending $f$, as required.
\end{proof}

\begin{lemma}\label{im6} $\ind$ has the full existence property over any set.
\end{lemma}
\begin{proof}   Assume that $A,B,C$  are subsets of $\c_\mathrm{Sq}$. We want to find some $A^\prime\equiv_C A$ such that  $A^\prime\ind_C B$. Without loss of generality $A$, $B$ and $C$ are closed (so in particular $C$ satisfies the hypotheses of 3(c) in Fact~\ref{im1})  and $C\subseteq A\cap B$. 
If $A = C$, then $A$ is algebraic over $C$ and we can take $A^\prime = A$. If $A\neq C$, then no element of $A$ is algebraic over $C$. Let $\overline{a}$ enumerate $A \setminus C$. By P.M. Neumann's Separation Lemma  (\cite{PJCpermutation}, Theorem~6.2), for every finite subtuple $\overline{a}_0$ of $\overline{a}$ there is $\overline{a}_0^\prime \equiv_C \overline{a}_0$ such that $\overline{a_0}^\prime \cap B = \emptyset$. It follows that there is $A^\prime\equiv_C A$  such that  $A^\prime\cap B= C$. By Remark~\ref{im2}, the partial STS $(A^\prime B,\P^{A^\prime B})$ can be extended to a Steiner quasigroup $(D,\cdot)$ which is freely generated by $A^\prime B$.  There is an embedding $f$  of $(D,\cdot)$ in $(\c_\mathrm{Sq},\cdot)$ over $B$.  If $A^{\prime\prime}= f(A^\prime)$, then $A^{\prime\prime}\equiv_B A^\prime$ and hence $A^{\prime\prime}\equiv_C A$. Since $f(D)$ is freely generated by $A^{\prime\prime} B$ and $A^{\prime\prime}\cap B=C$, we have $A^{\prime\prime}\ind_C B$.
\end{proof}

\begin{lemma}\label{im7}  $\ind$ has the stationarity property.
\end{lemma}
\begin{proof} Assume   $A,A^\prime,B,C$ are closed, $A\equiv_C A^\prime$, $C\subseteq A\cap B$, $A\ind_C B$ and $A^\prime\ind_C B$. We check that $A\equiv_B A^\prime$.
By hypothesis   $\langle AB\rangle$ is freely generated from $AB$,  and $\langle A^\prime B\rangle$ is freely generated from $A^\prime B$.  Fix some isomorphism of  STSs $f:A\rightarrow A^\prime$ over $C$, let $id_B$ be the identity mapping on $B$ and notice that  $f\cup id_B: AB\rightarrow A^\prime B$ is an isomorphism of partial STSs.  By the uniqueness of freely generated Steiner quasigroups, $f\cup id_B$  extends to some isomorphism of Steiner quasigroups $g: \langle AB\rangle \rightarrow \langle A^\prime B\rangle$ that witnesses $A\equiv_B A^\prime$.
\end{proof}

\begin{lemma}\label{im8} $\ind$ satisfies the following weak version of the freedom property: if $A,B,C,D$  are closed, $C\cap(\langle AD\rangle \langle BD\rangle) \subseteq D\subseteq C$ and $A\ind_C B$, then  $A\ind_DB$.
\end{lemma}
\begin{proof}  The assumption $A\ind_C B$ implies that $\langle AC\rangle \cap \langle BC\rangle = C$ and that $\langle ABC\rangle $ is freely generated by $\langle AC\rangle \langle BC\rangle$.  Notice that $\langle AD\rangle\cap \langle BD\rangle = D$.  We check that $\langle ABD\rangle$  is freely generated by $\langle AD\rangle\langle BD\rangle$. Let $f:\langle AD\rangle \langle BD\rangle\rightarrow E$ be a homomorphism of the partial STS $(\langle AD\rangle \langle BD\rangle,\P^{\langle AD\rangle \langle BD\rangle})$ to the Steiner quasigroup $(E,\cdot)$ and let us check that $f$ can be extended to $\langle ABD\rangle$.    By Lemma~\ref{extension} there is a Steiner quasigroup $(E^\prime,\cdot)\supseteq (E,\cdot)$ and a homomorphism $f^\prime\supseteq f$ from $(\langle AC\rangle\langle BC\rangle , \P^{\langle AC\rangle\langle BC\rangle})$ to $(E^\prime,\cdot)$.  Since $\langle ABC\rangle $ is freely generated by $\langle AC\rangle \langle BC\rangle$, we can extend $f^\prime$ to a homomorphism $g: \langle ABC\rangle \rightarrow E^\prime$.  Since  $g(ABD)=f(ABD)\subseteq E$, it follows that $g(\langle ABD\rangle)\subseteq E$ and therefore $g\restriction \langle ABD\rangle$ is a homomorphism to $(E,\cdot)$, as required.
\end{proof}

We are now ready to prove that $T^\ast_\mathrm{Sq}$ has elimination of imaginaries. We use parts~1 and~2 of Fact~\ref{im1}, and a version of part~3 where freedom is replaced by the property in Lemma~\ref{im8}. Moreover, we should remove the requirement that $\ov{a}$ should enumerate a closed set and instead deal with the general case. Since our assumptions are slightly different from those in Conant's original result (\cite{Conant17}, Lemma~5.5), we repeat the proof and adapt it to our setting.

\begin{prop} $T^\ast_\mathrm{Sq}$ has elimination of hyperimaginaries and weak elimination of imaginaries.
\end{prop}
\begin{proof} 
Let $\ov{a}_E$ be a hyperimaginary and let $\ov{b}$ be a minimal tuple in $\Sigma(\ov{a},E)$. Part~2~of Fact~\ref{im1} gives that $\ov{b}\in \bdd(\ov{a}_E)$. So it suffices to check that $\ov{a}_E\in \dcl(\ov{b})$. In fact, this holds for any element of $\Sigma(\ov{a},E)$.

Suppose that $\ov{a}$ is closed  and let  $\ov{c}\in\Sigma(\ov{a},E)$. Let $f$ be an automorphism of the monster model fixing $\ov{c}$ and let us check that  $E(\ov{a},f(\ov{a}))$. This will show that $f(\ov{a}_E)= \ov{a}_E$ and hence that $\ov{a}_E\in\dcl(\ov{c})$. By definition of $\Sigma(\ov{a},E)$, there is an indiscernible sequence $I=(\ov{a}_i\mid i<\omega)$ with $\ov{a}=\ov{a}_0$, with common intersection $\ov{c}$  and such that $E(\ov{a}_i,\ov{a}_j)$  for all $i,j$. It follows that $\ov{c}$ is closed and $I$ is $\ov{c}$-indiscernible. By Lemma~\ref{im6} (full existence) there is some $\ov{b}$ such that $\ov{b}\equiv_{\ov{a}}\ov{a}_1$  and  $\ov{b}\ind_{\ov{a}}\ov{a}_1$.  Notice that $E(\ov{b},\ov{a})$ and hence $E(\ov{b},\ov{a}_1)$.  Since  $\ov{b}\ov{a}\equiv_{\ov{c}}\ov{a}_1\ov{a}$, we have $\ov{b}\cap\ov{a}= \ov{c}=\ov{a}\cap\ov{a}_1$. By symmetry $\ov{a}_1\ind_{\ov{a}} \ov{b}$.
Since $\ov{a}\cap (\ov{b}\ov{a}_1) = \ov{c}\subseteq \ov{a}$ and $\ov{b}\ind_{\ov{a}}\ov{a}_1$, we may apply Lemma~\ref{im8} and obtain $\ov{b}\ind_{\ov{c}}\ov{a}_1$.
 Another application of Lemma~\ref{im6} gives some $\ov{b}_0\equiv_{\ov{c}}\ov{a}$ such that  $\ov{b}_0\ind_{\ov{c}}\ov{a}f(\ov{a})$. 
Since $\ov{a}\equiv_{\ov{c}}\ov{a}_1$, there is some $\ov{b}_1$ such that $\ov{b}_0\ov{a}\equiv_{\ov{c}}\ov{b}_1\ov{a}_1$.
Lemma~\ref{im5} (monotonicity) and invariance gives $\ov{b}_1\ind_{\ov{c}}\ov{a}_1$. Since $\ov{b}_1\equiv_{\ov{c}}\ov{b}$, by Lemma~\ref{im7} (stationarity) we have $\ov{b}\ov{a}_1\equiv_{\ov{c}} \ov{b}_1\ov{a}_1\equiv_{\ov{c}}\ov{b}_0 \ov{a}$. Since   $E(\ov{b},\ov{a}_1)$  we have  that $E(\ov{b}_0,\ov{a})$.  By symmetry and monotonicity, $\ov{a}\ind_{\ov{c}} \ov{b}_0$ and  $f(\ov{a})\ind_{\ov{c}} \ov{b}_0$. Since $f$ fixes $\ov{c}$,  we have $\ov{a}\equiv_{\ov{c}}f(\ov{a})$. By stationarity,  $\ov{a}\ov{b}_0\equiv_{\ov{c}}f(\ov{a})\ov{b}_0$, which implies that $E(\ov{b}_0,f(\ov{a}))$ and therefore $E(\ov{a},f(\ov{a}))$.

 Now we consider the general case where $\ov{a}$ might  not be closed.  We argue as in the proof of Theorem~5.6 of~\cite{Conant17}: take an enumeration $\ov{a}^\prime$ of the rest of $\langle \ov{a}\rangle$, define $E^\prime(\ov{x},\ov{x}^\prime;\ov{y},\ov{y}^\prime) \leftrightarrow E(\ov{x},\ov{y})$ and observe that  $\ov{a}_E$  and  $(\ov{a},\ov{a}^\prime)_{E^\prime}$ are interdefinable.  We know that there is a tuple $\ov{b}$  such that   $(\ov{a},\ov{a}^\prime)_{E^\prime}\in\dcl(\ov{b})$  and  $\ov{b}\in\bdd( (\ov{a},\ov{a}^\prime)_{E^\prime})$. It follows that $\ov{a}_E\in\dcl(\ov{b})$ and  $\ov{b}\in\bdd(\ov{a}_E)$, and hence $\ov{a}_E$ is eliminable.
\end{proof}

\textbf{Acknowledgements} The first author would like to thank Katie Chicot and Bridget Webb for preliminary discussion of Steiner triple systems. The authors are grateful to Martin Ziegler for helpful suggestions, and to the anonymous referee for a thorough review and numerous helpful comments.
\nocite{ChernikovRamsey16}


\begin{thebibliography}{21}

\bibitem{andersenetal}
{\sc Andersen, L., Hilton, A., and Mendelsohn, E.}
\newblock Embedding partial {S}teiner triple systems.
\newblock {\em Proceedings of the London Mathematical Society 41\/} (1980),
  557--576.
  
 \bibitem{balpao}
{\sc Baldwin, J. and Paolini, G.}
\newblock Strongly minimal {S}teiner systems I: existence.
\newblock 	arXiv:1903.03541, March 2019.

\bibitem{PJCinfinitels}
{\sc Cameron, P.}
\newblock Infinite linear spaces.
\newblock {\em Discrete Mathematics 129\/} (1994), 29--41.

\bibitem{PJCorbit}
{\sc Cameron, P.}
\newblock An orbit theorem for {S}teiner triple systems.
\newblock {\em Discrete Mathematics 125\/} (1994), 97--100.

\bibitem{PJCpermutation}
{\sc Cameron, P.}
\newblock {\em Permutation {G}roups}, volume 45 of
{\em London Mathematical Society Student Texts}. 
\newblock Cambridge University Press, 1999.

\bibitem{cameronwebb}
{\sc Cameron, P., and Webb, B.}
\newblock Perfect countably infinite {S}teiner triple systems.
\newblock {\em Australasian Journal of Combinatorics 54\/} (2012), 273--278.

\bibitem{Casanovas11}
{\sc Casanovas, E.}
\newblock {\em Simple {T}heories and {H}yperimaginaries}, volume~39 of {\em Lecture
  Notes in Logic}.
\newblock Cambridge University Press, 2011.

\bibitem{Cha-Kei90}
{\sc Chang, C., and Keisler, H.}
\newblock {\em Model Theory}, third~ed.
\newblock North Holland, 1990.

\bibitem{Chernikov12}
{\sc Chernikov, A.}
\newblock Theories without the tree property of the second kind.
\newblock {\em Annals of Pure and Applied Logic 165\/} (2014), 695--723.

\bibitem{ChernikovKaplan09}
{\sc Chernikov, A., and Kaplan, I.}
\newblock Forking and dividing in ${NTP}_2$ theories.
\newblock {\em The Journal of Symbolic Logic 77}, 1 (2012), 1--20.

\bibitem{ChernikovRamsey16}
{\sc Chernikov, A., and Ramsey, N.}
\newblock On model-theoretic tree properties.
\newblock {\em Journal of Mathematical Logic 16}, 2 (2016).

\bibitem{chicotetal}
{\sc Chicot, K., Grannell, M., Griggs, T., and Webb, B.}
\newblock On sparse countably infinite {S}teiner triple systems.
\newblock {\em Journal of Combinatorial Designs 18}, 2 (2010), 115--122.

\bibitem{colbournrosa}
{\sc Colbourn, C., and Rosa, A.}
\newblock {\em Triple Systems}.
\newblock Oxford University Press, 1999.

\bibitem{Conant17}
{\sc Conant, G.}
\newblock An axiomatic approach to free amalgamation.
\newblock {\em The Journal of Symbolic Logic 82\/} (2017), 648--671.

\bibitem{ConantKruckman17}
{\sc Conant, G., and Kruckman, A.}
\newblock Independence in generic incidence structures.
\newblock arXiv:1709.09626, September 2017.

\bibitem{Doyen69}
{\sc Doyen, J.}
\newblock Sur la structure de certains syst\`emes triples de {S}teiner.
\newblock {\em Mathematische Zeitschrift 111\/} (1969), 289--300.

\bibitem{KaplanRamsey17}
{\sc Kaplan, I., and Ramsey, N.}
\newblock On {K}im-independence.
\newblock arXiv:1702.03894. To appear in the {J}ournal of the {E}uropean
  {M}athematical {S}ociety, July 2017.

\bibitem{KaplanRamseyShelah17}
{\sc Kaplan, I., Ramsey, N., and Shelah, S.}
\newblock Local character of {K}im-independence.
\newblock arXiv:1707.02902v2, July 2017.


\bibitem{lindner}
{\sc Lindner, C.}
\newblock A survey of embedding theorems for {S}teiner systems.
\newblock {\em Annals of Discrete Mathematics 7\/} (1980), 175--202.


\bibitem{Sh80a}
{\sc Shelah, S.}
\newblock Simple unstable theories.
\newblock {\em Annals of Mathematical Logic 19\/} (1980), 177--203.

\bibitem{thomas}
{\sc Thomas, S.}
\newblock Triangle transitive {S}teiner triple systems.
\newblock {\em Geometriae Dedicata 21\/} (1986), 19--20.


\bibitem{treash}
{\sc Treash, C.}
\newblock The completion of finite incomplete {S}teiner triple systems with
  applications to loop theory.
\newblock {\em Journal of Combinatorial Theory Series A 10\/} (1971).

\end{thebibliography}

\noindent{\sc
School of Mathematics and Statistics\\
 The Open University\\
{\tt silvia.barbina@open.ac.uk}

\noindent{\sc
Departament de Matem\`atiques i Inform\`atica\\
Universitat de Barcelona}\\
{\tt e.casanovas@ub.edu}\\

\end{document}